\documentclass[11pt]{amsart}
\usepackage[top=1in, bottom=1in, left=1in, right=1in]{geometry}

\usepackage{amssymb,amsmath,amsthm,mathtools}
\usepackage{mathrsfs} 
\usepackage{bbm} 
\usepackage{enumitem}
\usepackage[hyphens]{url}
\usepackage{times}

\usepackage{graphicx}
\usepackage{color}
\usepackage[colorlinks=true]{hyperref}   

\numberwithin{equation}{section}

\newtheorem{thm}{Theorem}[section]
\newtheorem{lem}[thm]{Lemma}

\newtheorem{conj}[thm]{Conjecture}
\newtheorem{cor}[thm]{Corollary}

\theoremstyle{definition}

\theoremstyle{remark}
\newtheorem{remark}{Remark}[section]
\newtheorem*{remark*}{Remark}

\newtheoremstyle{claim} 
    {1em}                    
    {1em}                    
    {}                   
    {}                           
    {\bfseries}                   
    {.}                          
    {.5em}                       
    {}  

\theoremstyle{claim}

\newcommand{\m}{\mathfrak{m}}

\newcommand{\bs}\boldsymbol{}

\DeclareMathOperator{\sym}{sym}

\DeclareMathOperator{\sgn}{sgn}

\DeclareMathOperator{\SL}{SL}

\newcommand{\brac}[1]{\langle#1\rangle}
\renewcommand{\tilde}{\widehat}
\renewcommand{\tilde}{\widetilde}

\renewcommand{\phi}{\varphi}

\renewcommand{\Re}{{\rm Re}}
\renewcommand{\Im}{{\rm Im}}

\renewcommand{\bar}[1]{\overline{#1}}

\renewcommand{\mod}[1]{\,({\rm mod}\,#1)}

\definecolor{red}{rgb}{1,0,0}

\definecolor{orange}{rgb}{0.7,0.3,0}

\definecolor{blue}{rgb}{.2,.6,.75}

\definecolor{green}{rgb}{.4,.7,.4}

\begin{document}

\title{Mass Distribution for holomorphic cusp forms on the vertical geodesic}

\author{Qingfeng Sun and Qizhi Zhang}
\address{School of Mathematics and Statistics, Shandong University, Weihai,
		Weihai, Shandong 264209, China}

\email{qfsun@sdu.edu.cn\\
qzzhang@mail.sdu.edu.cn}

\subjclass[2010]{}
%

\begin{abstract}
We compute the quantum variance of holomorphic cusp
forms on the vertical geodesic for smooth compactly supported test functions.
As an application we show that almost all holomorphic Hecke cusp forms,
whose weights are in a short interval, satisfy QUE
conjecture on the vertical geodesic.
\end{abstract}

\thanks{Q. Sun was partially
  supported by the National Natural Science Foundation
  of China (Grant Nos. 11871306 and 12031008) and the Natural
  Science Foundation of Shandong Province (Grant No. ZR2023MA003)}

\maketitle
\section{Introduction}
The \textit{Quantum Unique Ergodicity} (QUE) conjecture introduced by Rudnik and Sarnak in \cite{rudnickBehaviourEigenstatesArithmetic1994},
is concerned about the equidistribution of Hecke Maass eigenforms as the Laplace eigenvalue tends to infinity.
More precisely, suppose that there is a smooth compactly supported function $\psi$ on $X:=
{\rm SL_2(\mathbb{Z})}\backslash \mathbb{H}$ and
$f(z)$ is a $L^2$-normalized Hecke Maass cusp form with Laplace eigenvalue $\lambda$.
Then QUE predicts that
$$
\int_{X} \psi(z)|f(z)|^2 \frac{dxdy}{y^2} \to \frac{3}{\pi} \int_{X} \psi(z)
\frac{dxdy}{y^2}
$$
as $\lambda \to \infty$, which has been proved by Lindenstrauss \cite{lindenstraussInvariantMeasuresArithmetic2006} and
Soundararajan \cite{soundararajanQuantumUniqueErgodicity2010}. An analogous conjecture for holomorphic Hecke cusp forms
predicts the same thing by replacing $f(z)$ with $f(z)y^{k/2}$ and holding the same limit,
but at this time we take $k\rightarrow \infty$, where $k$ is the weight of the associated holomorphic cusp form.
This problem was solved by Holowinsky and Soundararajan in \cite{holowinskyMassEquidistributionHecke2010c}.

In \cite{youngQuantumUniqueErgodicity2016}, Young considered some analogs of the quantum unique ergodicity
conjecture for some “thin” sets, which are connected with the behavior of eigenfunctions when restricted to certain lower
dimensional submanifolds of interest, such as QUE restricted to geodesics, horocycles or shrinking discs.
Moreover, he showed that the Eisenstein series $E_T(z)=E(z, 1/2+iT)$ on the modular group
equi-distributes along the geodesic connecting 0 and $i\infty$, thus refining the results
for the QUE problem of the Eisenstein series on the modular group derived by
Luo and Sarnak \cite{luosarnakquantumergodicityofeigenfunctions}. In \cite{matteweisenstein2018}, he further proved that the
 Eisenstein series satisfies QUE when restricted to a vertical geodesic
 with $\Re(z)$ abritrary.
 In this direction, Young then put forward the following restricted QUE conjecture
 along geodesic segments for Hecke cusp forms.

\begin{conj}\cite[Conjecture 1.1]{youngQuantumUniqueErgodicity2016}\label{verticalgeo}
Suppose that $\psi\colon \mathbb{R}^+ \to \mathbb{R}$ is a smooth
compactly supported function. Then
$$
\lim_{k\to \infty} \int_0^\infty y^{k} |f(iy)|^2 \psi(y) \frac{dy}{y}
= \frac{3}{\pi} \int_0^\infty \psi(y)\frac{dy}{y},
$$
where $f(z)$ runs over weight $k$ holomorphic Hecke cusp forms that are $L^2$-normalized.
\end{conj}

As Young mentioned, Conjecture \ref{verticalgeo} is out of reach of the current technology,
because a sharp upper bound of the form $\int_0^\infty |f(iy)|^2y^k \frac{dy}{y} \ll k^\varepsilon$
would imply a subconvexity bound of the form $L(f, 1/2) \ll \mathfrak{Q}^{1/8+\varepsilon}$,
where $\mathfrak{Q}\asymp k^2$ is the analytic conductor of $L(s, f)$, while
a subconvexity of this strength is currently not even known for the
Riemann zeta function.

 The equidistribution problem for almost all eigenforms is often termed \textit{Quantum Ergodicity} (QE)
 in the literature. In \cite{Zenz2021QuantumVF}, Peter proved such a QE result for holomorphic Hecke cusp forms
 on the vertical geodesic. Precisely,
 let $f$ be a holomorphic Hecke cusp form of weight $k$ for the full modular group $\Gamma=\SL_2(\mathbb{Z})$
 on the upper half plane $\mathbb{H}$. We normalize $f$ such that $||f||_2^2=\brac{f,f}=1$, where
$$\brac{f,g}:=\int_{X}f(z)\overline{g(z)}y^{k}\frac{dxdy}{y^2}$$
is the standard Petersson inner product for two holomorphic cusp forms $f,g$ of weight $k$.
For a smooth compactly supported test function $\psi\colon \mathbb{R^+} \to \mathbb{R}$, we define
\begin{equation}\label{vertquant}
\mu_f(\psi)=\int_0^\infty |f(iy)|^2 y^{k} \psi(y)\frac{dy}{y} \quad \text{and} \quad \mathbb{E}(\psi)=\frac{3}{\pi}\int_0^\infty \psi(y)\frac{dy}{y}.
\end{equation}
Peter \cite{Zenz2021QuantumVF} proved that, for any fixed $\epsilon>0$,
$$
\#\big\{f \in \bigcup_{\substack{K <  k \leq 2K\\ k\equiv 0 \mod{2}}}
H_k: |\mu_f(\psi)-\mathbb{E}(\psi)| > K^{-1/4+\epsilon}\big\} =o(K^2),
$$
i.e., almost all eigenforms $f \in  \bigcup_{\substack{K <  k \leq 2K\\ k\equiv 0 \mod{2}}} H_k$
are equi-distributed.

In this paper, we will refine Peter's result by showing that
such an analogous result holds for holomorphic Hecke cusp forms on the vertical geodesic
with weight in short intervals.
Similarly as in \cite{Zenz2021QuantumVF}, to measure how far $\mu_f(\psi)$ deviates from its expected
value $\mathbb{E}(\psi)$, we introduce the so-called \textit{quantum variance} (for the vertical geodesic),
which is (up to a suitable normalization) given by
$$\sum_{G < k-K \leq 2G} \sum_{f\in H_k} \big| \mu_f(\psi)-\mathbb{E}(\psi)\big|^2,$$
where $G=K^{\theta}$ with $0<\theta<1$ to be decided later, and
$H_k$ is a basis of holomorphic Hecke cusp forms. We will prove the following result.
\begin{thm}\label{mainthm}
Let $G=K^{\theta}$ with $\theta=\frac{9}{10}$.
Let $\psi_1,\psi_2$ and $h$ be smooth compactly supported functions on $\mathbb{R}^+$.
Suppose that $\psi_i(y)=\psi_i(1/y)$ for $i=1,2$.
Then
\begin{equation}
\label{thmequ1}\sum_{k\equiv 0\mod{2}} h\left(\frac{k-1-K}{G}\right)
\sum_{f\in H_k} L(1, \sym^2 f)  \Big( \mu_f(\psi_1) - \mathbb{E}(\psi_1)\Big)\cdot
\Big(\mu_f(\psi_2) - \mathbb{E}(\psi_2)\Big) = V(\psi_1, \psi_2)
\end{equation}
with
\begin{align*}
V(\psi_1&, \psi_2)=K^{\frac{1}{2}}G\log K\cdot \frac{\sqrt{2}\pi}{32}\tilde{\psi_1}(0)\tilde{\psi_2}(0) \int_0^\infty \frac{h(\sqrt{u})}{\sqrt{2\pi u}}du\\
+&K^{\frac{1}{2}}G\cdot\int_0^\infty \frac{h(\sqrt{u})}{\sqrt{2\pi u}} du \Big(\frac{\sqrt{2}\pi}{16}\Big(\frac{3}{2} \gamma - \log (4\pi)\Big) \tilde{\psi_1}(0)\tilde{\psi_2}(0) +\frac{\sqrt{2}\pi}{16} \tilde{\psi_1}(0)\tilde{\psi_2}'(0) \Big) \\
+&K^{\frac{1}{2}}G\cdot\frac{\sqrt{2}\pi}{16} \int_0^\infty \frac{h(\sqrt{u})}{\sqrt{2\pi u}} du \cdot  \frac{1}{2\pi i} \int_{(1)} \tilde{\psi}_1(-s_2)\tilde{\psi_2}(s_2)\zeta(1-s_2)\zeta(1+s_2)ds_2\\
+&O_{\psi_1,\psi_2}(K^{\frac{11}{8}}).\nonumber
 \end{align*}
as $K \to \infty$.
\end{thm}
In the computation of quantum variances one usually chooses a test function $\psi$
which is on average $0$, i.e. $\mathbb{E}(\psi)=0$.
As a corollary of Theorem \ref{mainthm} we have the following result.
\begin{cor}\label{maincor}
Let $G=K^{\theta}$ with $\theta=\frac{9}{10}$.
Let $\psi_1,\psi_2$ and $h$ be smooth compactly supported functions on $\mathbb{R}^+$.
Suppose that $\psi_i(y)=\psi_i(1/y)$ and $\mathbb{E}(\psi_i)=0$ for $i=1,2$. Then
\begin{equation}\label{thmequ2}
\sum_{k\equiv 0\mod{2}} h\left(\frac{k-1-K}{G}\right) \sum_{f\in H_k} L(1, \sym^2 f)   \mu_f(\psi_1) \mu_f(\psi_2) = V(\psi_1, \psi_2)
\end{equation}
with
\begin{align}\label{slickvar}
V(\psi_1, \psi_2)=&K^{\frac{1}{2}}G\cdot\frac{\sqrt{2}\pi}{16} \int_0^\infty \frac{h(\sqrt{u})}{\sqrt{2\pi u}} du \cdot  \frac{1}{2\pi i} \int_{(1)} \tilde{\psi}_1(-s_2)\tilde{\psi_2}(s_2)\zeta(1-s_2)\zeta(1+s_2)ds_2\\
&+O_{\psi_1,\psi_2}(K^{\frac{11}{8}}),\nonumber
\end{align}
as $K\to \infty$.
\end{cor}

A direct consequence of Theorem \ref{mainthm} is the following equi-distribution result.
\begin{cor}[Quantum Ergodicity]
Let $f$ be a holomorphic Hecke cusp form of even weight $k$ and $\psi$ a real-valued
smooth compactly supported on $\mathbb{R}^+$. For any $\epsilon>0$, we have
$$
\#\Bigg\{f \in \bigcup_{\substack{G <  k-K \leq 2G\\ k\equiv 0 \mod{2}}}
H_k: |\mu_f(\psi)-\mathbb{E}(\psi)| > K^{-1/4+\epsilon}\Bigg\} =o(KG),
$$
i.e., almost all eigenforms $f \in  \bigcup_{\substack{G <  k-K \leq 2G\\ k\equiv 0 \mod{2}}} H_k$ are equi-distributed.
\end{cor}

\section{Proof of Theorem \ref{mainthm}}
\subsection{Background}
Let $\psi$ and $h$ be smooth compactly supported functions on $\mathbb{R^+}$ and let $H_k$ denote a
basis of Hecke cusp forms of weight $k$. We are concerned with the asymptotic behavior of
the sum
\begin{equation*}
\label{startingpoint}V(\psi_1, \psi_2)=
\sum_{k\equiv 0\mod{2}} h\Big(\frac{k-1-K}{G}\Big) \sum_{f\in H_k} L(1, \sym^2 f)
\Big( \mu_f(\psi_1) - \mathbb{E}(\psi_1)\Big) \Big(\mu_f(\psi_2) - \mathbb{E}(\psi_2)\Big).
\end{equation*}

Let $f\in H_k$ be a normalized Hecke cusp form with the Fourier expansion
$$
f(z)=a_f(1)\sum_{n=1}^\infty \lambda_f(n) (4\pi n)^{\frac{k-1}{2}} e(nz),
$$
where $|a_f(1)|^2=\frac{2\pi^2}{\Gamma(k)L(1,\sym^2f)}$
arises from the $L_2-$normalization $||f||_2^2=1$.
The associated $L$-function is defined as
$$L(s, \sym^2f)=\zeta(2s)\sum_{n=1}^\infty \frac{\lambda_f(n^2)}{n^s}$$
for $\Re(s)>1$.

Define the Mellin transform of $\psi$ by
$$\tilde{\psi}(s):=\int_0^\infty \psi(y^{-1})y^{s-1} dy.$$
Then $\tilde{\psi}$ is entire and it satisfies
$$
\tilde{\psi}(s)\ll_{a, b, j} (|s|+1)^{-j}, \qquad a \leq \Re(s) \leq b
$$
for any integer $j>0$.
We also have $\tilde{\psi}(s)=\tilde{\psi}(-s)$ for $\psi(y)=\psi(y^{-1})$.
The Mellin inversion yields
$$\psi(y)=\frac{1}{2\pi i} \int_{(\sigma)} \tilde{\psi}(s)y^s ds, \quad \text{for }\sigma >0, \; y >0.$$

By \cite{Zenz2021QuantumVF}, the expression of $\mu_f(\psi)$ by seperating the terms with $m=n$ from
those with $m\neq n$ is the following.
\begin{lem}We have
\begin{align}\mathcal{E}_\psi:=&\int_0^\infty \psi(y)y^k |a_f(1)|^2 \sum_{n=1}^\infty \lambda_f(n)^2 (4\pi n)^{k-1} e^{-4\pi ny} \frac{dy}{y} -\frac{3}{\pi}\int_0^\infty \psi(y)\frac{dy}{y}\label{errorterm}\\
=&\frac{2\pi^2}{L(1,\sym^2f)} \cdot{\frac{1}{2\pi i}} \int_{(1/2)}\tilde{\psi}(s-1) \frac{\zeta(s) L(s, \sym^2f)}{(4\pi)^{s} \zeta(2s)} \frac{\Gamma(k+s-1)}{\Gamma(k)} ds. \nonumber
\end{align}
\end{lem}
Then we have
\begin{align}
\mu_f(\psi)\coloneqq\mathcal{E}_\psi +\mathcal{S}_\psi +\frac{3}{\pi}\int_0^\infty \psi(y)\frac{dy}{y},
\end{align}		
where
\begin{align}\label{shiftedconvolutiondefi}
\mathcal{S}_\psi\coloneqq|a_f(1)|^2 \int_0^\infty \sum_{n\neq m}\lambda_f(n)\lambda_f(m) (16\pi^2 nm)^{\frac{k-1}{2}} e^{-2\pi (n+m)y}\psi(y)y^{\frac{k}{2}} dy.
\end{align}

\subsection{Reduction to a Shifted Convolution Problem}

Following Lemma $5.2$ of \cite{Zenz2021QuantumVF},
we have the folloing asymptotic formula of  $\mathcal{S}_\psi$ to
evaluate the variance $V(\psi_1, \psi_2)$.

\begin{lem}\label{shiftedconvolutionint} We have
\begin{align}\mathcal{S}_\psi=\frac{\pi}{2 L(1, \sym^2f)}
\sum_{\ell\neq 0} \sum_n \frac{\lambda_f(n)\lambda_f(n+\ell)}{\sqrt{n(n+\ell)}}
\exp\Big( -\frac{k\ell^2}{2(2n+\ell)^2}\Big) \psi\Big(\frac{k}{2\pi(2n+\ell)}\Big)
+O_\psi(k^{-1/2+\varepsilon}).
\end{align}
\end{lem}
\begin{remark}
Since $\psi$ is smooth compactly supported on $\mathbb{R}^+$,
we see that $(2n+\ell) \asymp k$. More precisely,
if $\text{supp}\;\psi \subset (\alpha^{-1}, \alpha)$, then $2\pi\alpha^{-1} k\leq (2n+\ell) \leq 2\pi\alpha k$,
noting that $\alpha$ is a parameter which is only related to $\psi$ and can be taken
suitably under the condition $\alpha>1$. The exponential factor limits
the size of $\ell$ as otherwise we have rapid decay.
It follows that $\ell \ll k^{1/2+\varepsilon}$ and therefore $n \asymp k$.
\end{remark}
Subsequently, we only need to consider the case when $m>n$ and thus $\ell>0$ as the case with $m<n$ is exactly the same upon relabelling the variables.
\subsection{Trace formula}
We will use the following averaged Petersson trace formula.
\begin{lem}\cite[Iwaniec, Luo, Sarnak, Lemma 10.1]{ILS00}
For any positive numbers $m,n$ we have
\begin{align}\label{Petersson}\sum_{k \equiv 0 \mod{2}} &2h\Big(\frac{k-1}{K}\Big) \frac{2\pi^2}{k-1}\sum_{f\in H_k}\frac{ \lambda_f(m)\lambda_f(n)}{L(1, \sym^2f)}=\\
=&\hat{h}(0)K1_{m=n} -\pi^{1/2}(mn)^{-1/4}K \Im \Big(e^{-2\pi i/8}\sum_{c=1}^\infty \frac{S(m,n;c)}{\sqrt{c}} e_c(2\sqrt{mn})\hbar\Big(\frac{cK^2}{8\pi\sqrt{mn}}\Big)\Big)+ \nonumber \\
&+ O\Big(\frac{\sqrt{mn}}{K^4} \cdot \int_{-\infty}^\infty v^4 |\hat{h}(v)|dv+1_{m=n}\Big), \nonumber
\end{align}
where $\hat{h}$ denotes the Fourier transform of $h$ and $\hbar(v)=\int_0^\infty \frac{h(\sqrt{u})}{\sqrt{2\pi u}} e^{iuv} du$.
\end{lem}
\begin{remark}\label{decay}
We denote $g(x)=h\left(x-\frac{K}{G}\right)$, then g is also a smooth compactly supported function,
and if we suppose $\text{supp}\;h\subset(a,b)$, then we have $\text{supp}\;g\subset\left(a+\frac{K}{G},b+\frac{K}{G}\right)$.
Integrating by parts many times shows that $\bar{g}(v) \ll_{A}v^{-A}$ for any $A>0$.

In particular, the second term on the right hand side of \eqref{Petersson} is
absorbed into the error term if $cK^2/\sqrt{mn}>K^\varepsilon$.
In our case $mn$ will be of size $K^4$ and so this effectively
restricts the range of $c$ to $c \ll K^\varepsilon$.
\end{remark}

\subsection{Variance Computation}\label{VarianceComputation}
Now we compute the main term of the variance $V(\psi_1, \psi_2)$, which is given by
the averaged shifted convolution sum
\begin{equation}\label{Variance}
\mathcal{M}(\psi_1, \psi_2)=\sum_{k \equiv 0 \mod{2}}g\Big(\frac{k-1}{G}\Big) \sum_{f \in H_k} L(1, \sym^2f) S_{\psi_1}S_{\psi_2}
\end{equation}
with
 \begin{align*}S_{\psi}=&\frac{\pi}{2L(1,\sym^2f)}\sum_{\ell}\sum_{n\in\mathbb{N^+}}\frac{\lambda_f(n)\lambda_f(n+\ell)}{\sqrt{n(n+\ell)}}\exp\Big(\frac{-k\ell^2}{2(2n+\ell)^2}\Big)\psi\Big(\frac{k}{2\pi (2n+\ell)}\Big)+O_\psi(k^{-1/2+\epsilon})\\
= &\frac{\pi}{2L(1,\sym^2f)}\sum_{\substack{\ell, d \\ d|\ell}} \sum_{n} \frac{\lambda_f(n(n+\ell/d))}{\sqrt{d^2 n(n+\ell/d)}} \exp\Big( \frac{-k \ell^2}{2(2nd+ \ell)^2}\Big)\psi \Big( \frac{k}{2\pi (2nd+\ell)}\Big) +O_\psi(k^{-1/2+\epsilon})\\
 = &\frac{\pi}{2L(1,\sym^2f)} \sum_{d\geq 1} \sum_{m\neq 0} \sum_{n\geq1} \frac{\lambda_f(n(n+m))}{d\big(n(n+m)\big)^{1/2}}\exp\Big(-\frac{km^2}{2(2n+m)^2}\Big)\psi\Big(\frac{k}{2\pi d (2n+m)}\Big)+O_\psi(k^{-1/2+\varepsilon}).
 \end{align*}
 \begin{remark}
 It's easy to observe that $d(2n+m) \asymp k$ and $(dm)^2\ll k^{1/2+\varepsilon}.$ More precisely,   $2\pi\alpha^{-1} k\leq d(2n+m) \leq 2\pi\alpha k$  by supposing  $\text{supp}\;\psi \subset (\alpha^{-1}, \alpha)$.
\end{remark}

After expanding $S_{\psi_1}, S_{\psi_2}$ we see that the main term of  $\mathcal{M}(\psi_1, \psi_2)$ is equal to
$$ \sum_{\substack{n_1,n_2\\ d_1,d_2 \\m_1,m_2}} \sum_{k\equiv 0\mod{2}}2g^*\Big(\frac{k-1}{G}\Big)  \frac{2\pi^2}{k-1} \sum_{f\in H_k}\frac{1}{L(1, \sym^2f)} \frac{\lambda_f(n_1(n_1+m_1))\lambda_f(n_2(n_2+m_2))}{d_1d_2\big(n_1(n_1+m_1)n_2(n_2+m_2)\big)^{1/2}}$$
with
\begin{align*}
g^*(t)&=g^*_{n_1,n_2,d_1,d_2,m_1,m_2,K}(t)\\
&=g(t)\frac{tG}{16}\psi_1\Big(\frac{tG+1}{2\pi d_1(2n_1+m_1)}\Big)\psi_2\Big(\frac{tG+1}{2\pi d_2(2n_2+m_2)}\Big)\exp\Big(-\frac{(tG+1)m_1^2}{2(2n_1+m_1)^2}-\frac{(tG+1)m_2^2}{2(2n_2+m_2)^2}\Big).
\end{align*}
We can now apply the averaged Petersson trace formula (Lemma \ref{Petersson}) so that $\mathcal{M}(\psi_1, \psi_2)= \mathcal{D} + \mathcal{OD}$ with the diagonal
\begin{align}\label{diagonalterm}\mathcal{D}=G\sum_{\substack{n_1,n_2\\ d_1,d_2 \\m_1,m_2}}\frac{1_{n_1(n_1+m_1)=n_2(n_2+m_2)}}{d_1d_2\big(n_1(n_1+m_1)n_2(n_2+m_2)\big)^{1/2}} \cdot \widehat{g^*}(0)
\end{align}
and the off-diagonal
\begin{align}
\mathcal{OD}=- \sqrt{\pi} G \Im\bigg(&e^{-\frac{\pi i}{4}}\sum_{\substack{n_1,n_2\\ d_1,d_2 \\m_1,m_2}} \sum_{c=1}^\infty \frac{S(n_1(n_1+m_1),n_2(n_2+m_2);c)}{\sqrt{c}} e_c(2\sqrt{n_1(n_1+m_1)n_2(n_2+m_2)})\cdot \\
\times&\frac{1}{d_1d_2 \big(n_1(n_1+m_1)n_2(n_2+m_2)\big)^{3/4}}\cdot  \bar{g}^*\Big(\frac{c G^2}{8\pi \sqrt{n_1(n_1+m_1)n_2(n_2+m_2)}}\Big)\bigg),\nonumber
\end{align}
where $\bar{g}^*(v)=\int_0^\infty \frac{g^*(\sqrt{u})}{\sqrt{2\pi u}}e^{iuv} du$.
Then we have
\begin{equation}\label{main}
\mathcal{M}(\psi_1,\psi_2)=\mathcal{D}+\mathcal{OD}.
\end{equation}

We first evaluate the diagonal term.
 \subsection{Evaluation of the diagonal term}
We wil prove the following result.
 \begin{lem}\label{diagonal} Let $\mathcal{D}$ be as in \eqref{diagonalterm}. We have
\begin{align*}
\mathcal{D}=&K^{\frac{1}{2}}G\log K\cdot \frac{\sqrt{2}\pi}{32}\tilde{\psi_1}(0)\tilde{\psi_2}(0) \int_0^\infty \frac{h(\sqrt{u})}{\sqrt{2\pi u}}\Big(1+\sqrt{u}\frac{G}{K}\Big)^{1/2} du \\
+&K^{\frac{1}{2}}G\cdot \frac{\sqrt{2}\pi}{32}\tilde{\psi_1}(0)\tilde{\psi_2}(0)\int_0^\infty \frac{h(\sqrt{u})}{\sqrt{2\pi u}}\Big(1+\sqrt{u}\frac{G}{K}\Big)^{1/2}\log\Big(1+\sqrt{u}\frac{G}{K}\Big) du+\\
+&K^{\frac{1}{2}}G\cdot\int_0^\infty \frac{h(\sqrt{u})}{\sqrt{2\pi u}}\Big(1+\sqrt{u}\frac{G}{K}\Big) du \Big(\frac{\sqrt{2}\pi}{16}\Big(\frac{3}{2} \gamma - \log (4\pi)\Big) \tilde{\psi_1}(0)\tilde{\psi_2}(0) +\frac{\sqrt{2}\pi}{16} \tilde{\psi_1}(0)\tilde{\psi_2}'(0) \Big)\\
+&K^{\frac{1}{2}}G\cdot\frac{\sqrt{2}\pi}{16} \int_0^\infty \frac{h(\sqrt{u})}{\sqrt{2\pi u}} \Big(1+\sqrt{u}\frac{G}{K}\Big)^{1/2}du \cdot  \frac{1}{2\pi i} \int_{(1)} \tilde{\psi}_1(-s_2)\tilde{\psi_2}(s_2)\zeta(1-s_2)\zeta(1+s_2)ds_2\\
 +&O_{\psi_1, \psi_2}(K^{2+\varepsilon}G^{-1}),
\end{align*}
when $G=K^{\theta}$ is suffeciently large and $\theta>\frac{3}{4}$ holds.

Moreover,
\begin{align*}
\mathcal{D}=&K^{\frac{1}{2}}G\log K\cdot \frac{\sqrt{2}\pi}{32}\tilde{\psi_1}(0)\tilde{\psi_2}(0) \int_0^\infty \frac{h(\sqrt{u})}{\sqrt{2\pi u}} du \\
+&K^{\frac{1}{2}}G\cdot\int_0^\infty \frac{h(\sqrt{u})}{\sqrt{2\pi u}} du \Big(\frac{\sqrt{2}\pi}{16}\Big(\frac{3}{2} \gamma - \log (4\pi)\Big) \tilde{\psi_1}(0)\tilde{\psi_2}(0) +\frac{\sqrt{2}\pi}{16} \tilde{\psi_1}(0)\tilde{\psi_2}'(0) \Big)\\
+&K^{\frac{1}{2}}G\cdot\frac{\sqrt{2}\pi}{16} \int_0^\infty \frac{h(\sqrt{u})}{\sqrt{2\pi u}} du \cdot  \frac{1}{2\pi i} \int_{(1)} \tilde{\psi}_1(-s_2)\tilde{\psi_2}(s_2)\zeta(1-s_2)\zeta(1+s_2)ds_2\\
 +&O_{\psi_1, \psi_2}(K^{2+\varepsilon}G^{-1}+K^{-\frac{1}{2}+\varepsilon}G^{2}),
\end{align*}
 \end{lem}
 \begin{proof}
 To evaluate $\mathcal{D}$ (see \eqref{diagonalterm}), we first show that most solutions
 to $n_1(n_1+m_1)=n_2(n_2+m_2)$ arise from the diagonal, i.e., $n_1=n_2$ and $m_1=m_2$.
 In the proof of  \cite[Lemma 5.4]{Zenz2021QuantumVF}, Peter has shown that
 there are at most $K^{1+2\varepsilon+\varepsilon'}$ off-diagonal terms.
 These off-diagonal terms contribute to $\mathcal{D}$ at most
\begin{align*}
G \sum_{\substack{n_1, n_2, m_1, m_2\\ n_1\neq n_2, m_1\neq m_2}}
\sum_{d_1, d_2} \frac{1_{n_1(n_1+m_1)=n_2(n_2+m_2)}}{d_1d_2\big(n_1(n_1+m_1)n_2(n_2+m_2)\big)^{1/2}}
\cdot \widehat{g^*}(0) \ll_{\psi_1, \psi_2} G K^{2\varepsilon +\varepsilon'}.
\end{align*}
Hence,
\begin{align*}
\mathcal{D}=&\frac{\sqrt{2\pi}G^2}{16} \sum_{n, d_1, d_2, m}\frac{1}{d_1d_2 n(n+m)}\cdot \\
&\times \int_0^\infty \frac{g(\sqrt{u})\sqrt{u}}{\sqrt{2\pi u}}\psi_1\Big(\frac{\sqrt{u}G+1}{2\pi d_1(2n+m)}\Big)\psi_2\Big(\frac{\sqrt{u}G+1}{2\pi d_2(2n+m)}\Big) \exp\Big(- \frac{(\sqrt{u}G+1) m^2}{(2n
+m)^2}\Big) du\\
&+O_{\psi_1, \psi_2}( G K^{2\varepsilon +\varepsilon'}).
\end{align*}

Since $m_i d_i \ll K^{1/2+\varepsilon}$ and $m_i/n_i \ll K^{-1/2+\varepsilon}$ for $i=1,2$,
we can simplify the expression for the off-diagonal $\mathcal{D}$ by Taylor expansion.
Notice that we only do this when $\sqrt{u} \in \text{supp}\; g$, as otherwise $g$ becomes $0$. We compute
\begin{align*}
\psi_i\Big(\frac{\sqrt{u}G+1}{2\pi d_1(2n+m)}\Big)
&=\psi_i\Big(\frac{\sqrt{u}G}{4\pi nd_1}\Big)+O\left(\frac{m}{n}\right), \\
\exp\Big(- \frac{(\sqrt{u}G+1) m^2}{(2n
+m)^2}\Big)&=\exp\Big(- \frac{\sqrt{u}G m^2}{4n
^2}\Big)+O\left(\frac{m}{n}\right).
\end{align*}
Thus \eqref{diagonalterm} becomes
\begin{align*}\mathcal{D}=&\frac{\sqrt{2\pi}G^2}{16} \sum_{n, d_1,d_2, m} \frac{1}{d_1d_2n^2} \int_0^\infty \frac{ g(\sqrt{u})\sqrt{u}}{\sqrt{2\pi u}} \psi_1\Big(\frac{\sqrt{u}G}{4\pi nd_1}\Big)\psi_2\Big(\frac{\sqrt{u}G}{4\pi n d_2}\Big)\exp\Big(-\frac{\sqrt{u}Gm^2}{4n^2}\Big)du\\
&+O_{\psi_1, \psi_2}(G K^{\varepsilon}).
\end{align*}
To evaluate the main term of $\mathcal{D}$ asymptotically we perform
an  Mellin inversion on the smooth compactly supported functions $\psi_1, \psi_2$
and the exponential function and then shift the contours. The main term of $\mathcal{D}$ is equal to
\begin{align*}
\frac{\sqrt{2\pi} G^2}{16} \frac{1}{(2\pi i)^3}& \int_{(1/2+\varepsilon)}\int_{(1)}\int_{(1)} \int_0^\infty \frac{g(\sqrt{u})\sqrt{u}}{\sqrt{2\pi u}}\tilde{\psi}_1(s_1)\tilde{\psi_2}(s_2)\Gamma(s_3) \\
&\times \sum_{n_1, d_1,d_2,m_1} \frac{1}{d_1d_2n_1^2}\Big(\frac{\sqrt{u}G}{4\pi n_1 d_1}\Big)^{s_1}\Big(\frac{\sqrt{u}G}{4\pi n_1d_2}\Big)^{s_2}\Big(\frac{4n_1^2}{\sqrt{u}Gm_1^2}\Big)^{s_3} du ds_1ds_2ds_3.
\end{align*}
This in turn can be rewritten as
\begin{align*}\frac{\sqrt{2\pi} G^2}{16} \frac{1}{(2\pi i)^3}& \int_{(1/2+\varepsilon)}\int_{(1)}\int_{(1)} \int_0^\infty \frac{g(\sqrt{u})\sqrt{u}}{\sqrt{2\pi u}}u^{\frac{s_1+s_2-s_3}{2} } G^{s_1+s_2-s_3} (4\pi)^{-s_1-s_2}4^{s_3}\tilde{\psi}_1(s_1)\tilde{\psi_2}(s_2)\\
&\times \Gamma(s_3)\zeta(1+s_1)\zeta(1+s_2)\zeta(2+s_1+s_2-2s_3)\zeta(2s_3) duds_1ds_2ds_3.
\end{align*}

We start by shifting the contour from $\Re(s_3)=1/2+\varepsilon$ to $\Re(s_3)=100$.
The integral on the new line $\Re(s_3)=100$ is negligible by
the rapid decay of $\tilde{\psi_1}(s_1), \tilde{\psi_2}(s_2)$
and the Gamma function (it contributes no more than $O_{\psi_1, \psi_2}(G^{-10})$). The simpe pole at $s_3=1/2+s_1/2+s_2/2$ yields the residue
\begin{align}\label{diagonalstep4}\frac{\sqrt{2\pi} G^2}{16} \frac{1}{(2\pi i)^2}&\int_{(1)}\int_{(1)} \int_0^\infty \frac{g(\sqrt{u})\sqrt{u}}{\sqrt{2\pi u}}u^{\frac{-1+s_1+s_2}{4} }  G^{\frac{-1+s_1+s_2}{2} }   (4\pi)^{-s_1-s_2}4^{\frac{1+s_1+s_2}{2} } \\
&\times \tilde{\psi}_1(s_1)\tilde{\psi_2}(s_2)\Gamma(\frac{1+s_1+s_2}{2} )\zeta(1+s_1)\zeta(1+s_2)\frac{1}{2}\zeta(1+s_1+s_2)duds_1ds_2. \nonumber
\end{align}

Next we move the line $\Re(s_1)=1$ to $\Re(s)=-2+\varepsilon$ (stopping just before the
pole of the Gamma function), picking up simple poles at $s_1=0$ and $s_1=-s_2$ .
We use again the rapid decay of $\tilde{\psi_1}(s_1), \tilde{\psi_2}(s_2)$ to show that
the integral over the new line is bounded by $O_{\psi_1, \psi_2}(K^{2+\varepsilon}G^{-1})$.
At $s_1=0$ the residue is
\begin{align}\label{maintermdiagonal}\frac{\sqrt{2\pi} G^\frac{3}{2}}{16} \frac{1}{2\pi i} \int_{(1)}& \int_0^\infty \frac{g(\sqrt{u})u^{\frac{1}{4}}}{\sqrt{2\pi u}}u^{\frac{s_2}{4} } G^{\frac{s_2}{2}}  (2\pi)^{-s_2}  \tilde{\psi}_1(0)\tilde{\psi_2}(s_2)\Gamma(\frac{1+s_2}{2})\zeta(1+s_2)^2duds_2 .
\end{align}
We follow up with the shift from $\Re(s_2)=1$ to $\Re(s_2)=-1+\varepsilon$ and pick up
a double pole at $s_2=0$. The error term from the line at $\Re(s_2)=-1+\varepsilon$
is $O_{\psi_1, \psi_2}(G^{1+\varepsilon})$, which is no more than $K^{2+\varepsilon}G^{-1}$. To compute the residue at the double pole we use the expansion
$\zeta(1+s_2)^2=\frac{1}{s_2^2}+\frac{2\gamma}{s_2}+\cdots$ for $s_2$ close to $0$, where $\gamma$ is the Euler-Mascheroni constant. The residue is then given by
\begin{align*}\frac{\sqrt{2\pi} G^{\frac{3}{2}}}{16} \cdot  \int_0^\infty &\frac{g(\sqrt{u})u^{\frac{1}{4}}}{\sqrt{2\pi u}}\tilde{\psi}_1(0) \lim_{s_2 \to 0} \frac{d}{ds_2} \Big(s_2^2 \Big(\frac{G\sqrt{u}}{4\pi^2}\Big)^{\frac{s_2}{2}} \tilde{\psi_2}(s_2)\Gamma(\frac{1+s_2}{2}) \cdot \Big(\frac{1}{s_2^2}+\frac{2\gamma}{s_2  }+\cdots \Big)\Big)du.
\end{align*}
The limit equals
\begin{align*}\lim_{s_2\to 0} \Big(\frac{G\sqrt{u}}{4\pi^2}\Big)^{\frac{s_2}{2}}\cdot \Big( &\frac{1}{2} \log \Big(\frac{G\sqrt{u}}{4\pi^2}\Big)\tilde{\psi_2}(s_2)\Gamma(\frac{1+s_2}{2}) + \\
&+ \tilde{\psi_2}'(s_2)\Gamma(\frac{1+s_2}{2}) +\frac{1}{2}\tilde{\psi_2}(s_2)\Gamma'(\frac{1+s_2}{2})+2\gamma \tilde{\psi_2}(s_2)\Gamma(\frac{1+s_2}{2})\Big).
\end{align*}
Evaluating the limit, using $\Gamma'(1/2)=\sqrt{\pi} (-\gamma-\log4)$,  the value of this limit is
$$\frac{\sqrt{\pi}}{2}\tilde{\psi_2}(0)  \log G  +  \frac{\sqrt{\pi}}{4}\tilde{\psi_2}(0) \log u +\sqrt{\pi}\tilde{\psi_2}(0) \Big(\frac{3}{2} \gamma - \log (4\pi)\Big) + \sqrt{\pi}\tilde{\psi_2}'(0).$$
Thus \eqref{maintermdiagonal} is equal to
\begin{align*}
G^{\frac{3}{2}}&\log G\cdot \frac{\sqrt{2}\pi}{32}\tilde{\psi_1}(0)\tilde{\psi_2}(0)  \cdot \int_0^\infty \frac{g(\sqrt{u})u^{\frac{1}{4}}}{\sqrt{2\pi u}} du
+G^{\frac{3}{2}} \frac{\sqrt{2}\pi}{64}\tilde{\psi_1}(0)\tilde{\psi_2}(0)\int_0^\infty \frac{g(\sqrt{u})u^{\frac{1}{4}}}{\sqrt{2\pi u}} \log(u) du\\
+&G^{\frac{3}{2}}\int_0^\infty \frac{g(\sqrt{u})u^\frac{1}{4}}{\sqrt{2\pi u}} du \cdot \Big(\frac{\sqrt{2}\pi}{16}\Big(\frac{3}{2} \gamma - \log (4\pi)\Big) \tilde{\psi_1}(0)\tilde{\psi_2}(0) +\frac{\sqrt{2}\pi}{16} \tilde{\psi_1}(0)\tilde{\psi_2}'(0) \Big),
\end{align*}
Combining with equation $g\left(x\right)=h\left(x-\frac{K}{G}\right)$, we see that
\eqref{maintermdiagonal} equals
\begin{align*}
K&^{\frac{1}{2}}G\log K\cdot \frac{\sqrt{2}\pi}{32}\tilde{\psi_1}(0)\tilde{\psi_2}(0) \int_0^\infty \frac{h(\sqrt{u})}{\sqrt{2\pi u}}\Big(1+\sqrt{u}\frac{G}{K}\Big)^{1/2} du \\
+&K^{\frac{1}{2}}G\cdot \frac{\sqrt{2}\pi}{32}\tilde{\psi_1}(0)\tilde{\psi_2}(0)\int_0^\infty \frac{h(\sqrt{u})}{\sqrt{2\pi u}}\Big(1+\sqrt{u}\frac{G}{K}\Big)^{1/2}\log\Big(1+\sqrt{u}\frac{G}{K}\Big) du\\
+&K^{\frac{1}{2}}G\cdot\int_0^\infty \frac{h(\sqrt{u})}{\sqrt{2\pi u}}\Big(1+\sqrt{u}\frac{G}{K}\Big) du \Big(\frac{\sqrt{2}\pi}{16}\Big(\frac{3}{2} \gamma - \log (4\pi)\Big) \tilde{\psi_1}(0)\tilde{\psi_2}(0) +\frac{\sqrt{2}\pi}{16} \tilde{\psi_1}(0)\tilde{\psi_2}'(0) \Big).
\end{align*}
 There is another term of size $K^{\frac{1}{2}}G$ coming from the residue
 of \eqref{diagonalstep4} at $s_1=-s_2$. This residue is given by
\begin{align}\label{diagonalstep6}K^{\frac{1}{2}}&G\cdot\frac{\sqrt{2}\pi}{16}
\int_0^\infty \frac{h(\sqrt{u})}{\sqrt{2\pi u}} \Big(1+\sqrt{u}\frac{G}{K}\Big)^{1/2}du
\cdot  \frac{1}{2\pi i} \int_{(1)} \tilde{\psi}_1(-s_2)\tilde{\psi_2}(s_2)
\zeta(1-s_2)\zeta(1+s_2)ds_2.
\end{align}
\end{proof}
\subsection{Auxiliary Lemmas}
In the following section we record some lemmas that we use to compute the off-diagonal term asymptotically. We start with some observations regarding the function $\hbar(v)$, appearing in the off-diagonal term. For any complex number $w$ define the function
$$\hbar^{\Re}_{w}(v)=\int_0^\infty \frac{h(\sqrt{u})}{\sqrt{2\pi u}} u^{w/2}\cos(uv)du.$$
For $w=0$ this is the real part of $\hbar(v)$. The Mellin transform of this function and its properties were evaluated by Khan \cite[Lemma 3.5]{khanNonvanishingSymmetricSquare2010} (and also Das-Khan \cite[sec. 2.6]{dasThirdMomentSymmetric2018}). Note that Khan and Das-Khan treat $\hbar_w(v)$ but the observations also go through  for the real part that we consider. As in \cite[sec. 2.6]{dasThirdMomentSymmetric2018} we have by repeated integration by parts the bound
\begin{equation}\label{hbarreal}
\frac{\partial^j}{\partial v^j} \hbar^{\Re}_{w}(v) \ll (1+|w|)^A|v|^{-A}
\end{equation} for any non-negative integer $j,A$ and the implied constant
depends on $\Re(w), j, A$.
We denote the Mellin transform of $\hbar^{\Re}_{w}$ by
$$\tilde{\hbar}^{\Re}_w(s)=\int_0^\infty \hbar^{\Re}_w(v)v^s \frac{dv}{v}.$$
The estimate \eqref{hbarreal} implies that the Mellin transform is absolutely
convergent and holomorphic for $\Re(s)>0$. Integrating by parts several times
and using again the bound \eqref{hbarreal} show that the Mellin transform deacys rapidly.
More precisely, we have
$$\tilde{\hbar}^{\Re}_w(s) \ll (1+|w|)^{A+\Re(s)+1}(1+|s|)^{-A},$$
with the implied constant depending on $\Re(w)$ and $A$.

By Mellin inversion we have, for $c>0$,
\begin{equation}\label{inverseMellin}\hbar^{\Re}_w(v)= \frac{1}{2\pi i}\int_{(c)}\tilde{\hbar}^{\Re}_w(s)v^{-s} ds.
\end{equation}
As in \cite[Lemma 3.5]{khanNonvanishingSymmetricSquare2010} we can explicitly evaluate
the Mellin transform of $\hbar^\Re_w$ within the range $0 < \Re(s) <1$. Then we get
$$
\tilde{\hbar}^{\Re}_w(s)=\int_0^\infty
\frac{h(\sqrt{u})}{\sqrt{2\pi u}}u^{w/2}\Gamma(s)\cos(\pi s/2) du.
$$

The next two lemmas will be useful to treat the exponential sum in the off-diagonal term.
\begin{lem}[Poisson summation]\label{Poisson}Let $f$ be a rapidly decaying smooth function. Then
$$
\sum_{n\equiv a\mod{c}} f(n)= \frac{1}{c} \sum_{n} \hat{f}\Big(\frac{n}{c}\Big)e_c(an),
$$
where $\hat{f}(\xi)=\int_{-\infty}^\infty f(x)e(-x\xi) dx$ denotes the Fourier transform of $f$.
\end{lem}
\begin{proof}
This follows immediately from the classical Poisson summation formula and noting that $n\equiv a\mod{c}$ is a shifted lattice of $\mathbb{Z}$.
\end{proof}
We detect cancellation in the off-diagonal term with the stationary phase method.
We quote the following results of Blomer, Khan and Young
(see Proposition 8.1 and a special case of Proposition 8.2
in \cite{blomerDistributionMassHolomorphic2013b}).
\begin{lem} \label{integrationbyparts}
 Let $Y \geq 1$, $X, Q, U, R > 0$,
and suppose that $h$ is a smooth function with support on $[\alpha, \beta]$, satisfying
\begin{equation*}
h^{(j)}(t) \ll_j X U^{-j}.
\end{equation*}
Suppose $f$ 
  is a smooth function on $[\alpha, \beta]$ such that
\begin{equation}
 |f'(t)| \geq R
\end{equation}
for some $R > 0$, and
\begin{equation}\label{diffh0}
f^{(j)}(t) \ll_j Y Q^{-j}, \qquad \text{for } j=2, 3, \dots.
\end{equation}
Then the integral $I$ defined by
\begin{equation*}
I = \int_{-\infty}^{\infty} h(t) e^{i f(t)} dt
\end{equation*}
satisfies
\begin{equation}
\label{eq:Ipartsbound}
 I \ll_A (\beta - \alpha) X \left[\left(\frac{QR}{\sqrt{Y}} \right)^{-A} + (RU)^{-A}\right].
\end{equation}
\end{lem}

\begin{lem}[Stationary phase] \label{stationaryphase}
Let $X,Y,V,V_1,Q>0$ and $Z:=Q+X+Y+V_1+1$, and assume that
$$Y \geq Z^{3/20},\quad V_1 \geq V\geq \frac{QZ^{1/40}}{Y^{1/2}}.$$
Suppose that $h$ is a smooth function on $\mathbb{R}$ with support on an interval $J$ of length $V_1$, satisfying
$$h^{(j)}(t)\ll_j XV^{-j}$$ for all $j \in \mathbb{N}_0$. Suppose $f$ is a smooth function on $J$ such that there exists a unique point $t_0\in J$ such that $f'(t_0)=0$, and furthermore
$$f''(t) \gg YQ^{-2}, \quad f^{(j)}(t)\ll_j YQ^{-j}, \quad \text{ for } j\geq 3 \text{ and } t \in J.$$
Then
$$\int_{-\infty}^\infty h(t)e^{2\pi if(t)}dt = e^{\sgn(f''(t_0))\cdot \frac{\pi i}{4}}\frac{e^{2\pi if(t_0)}}{\sqrt{|f''(t_0)|}} h(t_0)+ O\left(\frac{Q^{3/2}X}{Y^{3/2}}\cdot\left(V^{-2}+\frac{Y^{2/3}}{Q^2} \right)\right).$$
In particular, we also have the trivial bound
$$\int_{-\infty}^\infty h(t) e^{2\pi i f(t)} dt \ll \frac{XQ}{\sqrt{Y}}+1.$$
\end{lem}
\begin{proof}
See Proposition 8.2 in \cite{blomerDistributionMassHolomorphic2013b}, with $\delta=1/20$
and $A$ sufficiently large. We bounded the contribution of the non-leading terms in the asymptotic expansion
of \cite[Eq. 8.9]{blomerDistributionMassHolomorphic2013b}) trivially by
$$O\Big(\frac{Q^{3/2}X}{Y^{3/2}}\cdot\Big(V^{-2}+\frac{Y^{2/3}}{Q^2} \Big)\Big).$$
\end{proof}

The following formula can be found in \cite[Lemma 5.7]{Zenz2021QuantumVF}
which concerns about Kloosterman sum.
\begin{lem}\label{Kloostermansum}Let $S(a,b; c)$ denote the classical Kloosterman sum. Then
$$\sum_{\substack{a_1 \mod{c},\\ a_2 \mod{c}}} \sum_{\substack{b_1 \mod{c} ,\\ b_2 \mod{c}}} S(a_1(a_1+b_1),a_2(a_2+b_2);c)e_c(2a_1a_2+a_1b_2+a_2b_1)=c^3\phi(c),$$
where $\phi(c)$ is Euler's totient function.
\end{lem}
We also quote the following result (see \cite{di1857note}).
\begin{lem}[Fa\`a di Bruno's Formula]
Suppose $ p(t)$ and $q(t)$ are functions for which all necessary derivatives are defined, then
\begin{align*}
\frac{d^n}{d t^n}p(q(t))=\sum\frac{n!}{k_1!k_2!k_3!\cdots k_n!}p^{(m)}(q(t))\left(\frac{1}{1!}\frac{d}{dt}q(t)\right)^{k_1}\left(\frac{1}{2!}\frac{d^2}{dt^2}q(t)\right)^{k_2}\cdots\left(\frac{1}{n!}\frac{d^n}{dt^n}q(t)\right)^{k_n},
\end{align*}
 where the sum is over all different solutions in nonnegative integers $k_1,k_2,k_3,\dots,k_n$ of $k_1+2k_2+3k_3+\cdots+nk_n=n$, and  $k_1+k_2+k_3+\cdots+k_n=m$.
\end{lem}

\subsection{Preparing for Evaluating the 	Off-Diagonal}
Our goal is to obtain an asymptotic formula for the off-diagonal $\mathcal{OD}$ given by
 \begin{align}\label{offdiagonal2}
-\sqrt{\pi} \Im \bigg(&e^{-\frac{\pi i}{4}} G \mathop{\sum\sum \sum}_{m_j \in\mathbb{N}/\{0\},d_j,n_j\in\mathbb{N}^{+}\atop j\in\{1,2\}} \sum_{c=1}^{\infty} \frac{S(n_1(n_1+m_1),n_2(n_2+m_2);c)}{\sqrt{c}}e_c(2n_1n_2+n_1m_2+n_2m_1)\\
& \times \frac{e_c\big(f(n_1,n_2,m_1,m_2)\big)}{d_1d_2 \big(n_1(n_1+m_1)n_2(n_2+m_2)\big)^{\frac{3}{4}}} \cdot \varpi_{c,G} \begin{pmatrix}
d_1 &m_1 &n_1\\
d_2 &m_2& n_2
\end{pmatrix}\bigg), \nonumber
 \end{align}
 where
 $$
 f(n_1, n_2, m_1, m_2)=2\sqrt{n_1(n_1+m_1)n_2(n_2+m_2)}-2n_1n_2-n_1m_2-n_2m_1,
 $$
and
 \begin{align*}
 \varpi_{c,G} \begin{pmatrix}
d_1 &m_1 &n_1\\
d_2 &m_2& n_2
\end{pmatrix}&\coloneqq \bar{g}_{d_1, d_2}^*\Big(\frac{c G^2}{8\pi \sqrt{n_1(n_1+m_1)n_2(n_2+m_2)}}\Big)
 \\
 &=\frac{G}{16}\int_0^\infty \frac{g(\sqrt{u})\sqrt{u}}{\sqrt{2\pi u}} \psi_1\Big(\frac{\sqrt{u}G}{2\pi d_1(2n_1+m_1)}\Big)\psi_2\Big(\frac{\sqrt{u}G}{2\pi d_2(2n_2+m_2)}\Big)  \nonumber\\
 &\times  \exp\Big(- \frac{\sqrt{u}G m_1^2}{(2n_1+m_1)^2}- \frac{\sqrt{u}G m_2^2}{(2n_2+m_2)^2}\Big) e^{iu \frac{cG^2}{8\pi (n_1(n_1+m_1)n_2(n_2+m_2))^{1/2}}} du\nonumber\\
  &=\frac{K}{16}\int_0^\infty \frac{h(\sqrt{u})}{\sqrt{2\pi u}} \Big(1+\frac{G\sqrt{u}}{K}\Big)\psi_1\Big(\frac{\sqrt{u}G+K}{2\pi d_1(2n_1+m_1)}\Big)\psi_2\Big(\frac{\sqrt{u}G+K}{2\pi d_2(2n_2+m_2)}\Big)  \nonumber\\
 &\times  \exp\Big(- \frac{(\sqrt{u}G+K) m_1^2}{(2n_1+m_1)^2}- \frac{(\sqrt{u}G+K) m_2^2}{(2n_2+m_2)^2}\Big) e^{ic \frac{uG^2+2\sqrt{u}KG+K^2}{8\pi (n_1(n_1+m_1)n_2(n_2+m_2))^{1/2}}} du.\nonumber
 \end{align*}
Notice that $ \varpi_{c,G}$ is smooth compact function of  paramaters $d_j,m_j,n_j$ for $j=1,2$.

Integrating by parts several times shows that $\varpi\ll K\left(\frac{cKG}{n_1n_2}\right)^{-A}$ for any $A>0$, which implies that $\frac{cKG}{n_1n_2}\ll K^{\varepsilon}$ as otherwise $\varpi$ is negligible.
By the condition $n_id_i\asymp K$ and $G=K^{\theta}$,
we have the upper bound
$d_i\ll K^{\delta}$, where $\delta>\frac{1}{2}(1-\theta+\varepsilon)$.

For technical reasons, we expect a lower bound on $m_j,j=1,2$.
The terms with $m_i<K^{\eta}$ in $\mathcal{OD}$ will be estimated
trivially by (noting that $\frac{m_i}{n_i}\ll K^{-\frac{1}{2}}$)
\begin{align}\label{oderrorterm1}
 KG \mathop{\sum\sum \sum}_{d_j\ll K^{\delta},m_j\ll K^{\eta}\atop n_jd_j\asymp K,j\in\{1,2\}} \sum_{c=1}^{\infty} \frac{\tau(c)\left(n_1(n_1+m_1),n_2(n_2+m_2),c\right)^{\frac{1}{2}} }{d_1d_2 \big(n_1(n_1+m_1)n_2(n_2+m_2)\big)^{\frac{3}{4}}}  \left(\frac{cKG}{n_1n_2}\right)^{-A}\ll K^{\frac{3}{2}-\frac{1}{2}\theta+2\eta+\varepsilon}
\end{align}		
by choosing $A=1$ and $c\ll K^{2}$.
To make this term less than main term, we let
$K^{\frac{3}{2}-\frac{1}{2}\theta+2\eta+\varepsilon}\ll GK^{\frac{1}{2}}$
which implies $\theta>\frac{2}{3}+\frac{4}{3}\eta+\varepsilon.$

Next we work under the following conditions:
\begin{equation}\label{cond1}n_id_i \asymp K \quad \text{and}\quad d_i \leq K^{\delta}\quad \text{for } i=1,2,\end{equation}
\begin{equation}\label{cond2}K^{\eta} \leq m_i\leq K^{1/2+\varepsilon}d_i^{-1} \quad \text{for } i=1,2,\end{equation}
\begin{equation}\label{cond3}\frac{cKG}{n_1n_2} \ll K^\varepsilon \quad \text{or equivalently}\quad cd_1d_2\ll K^{1- \theta+\varepsilon}\end{equation}
for $G=K^{\theta}$, where $\theta<1$ and $\delta,\eta$ satisfy
\begin{equation}
\label{cond1111}
\delta>\frac{1}{2}(1-\theta+\varepsilon),
\quad\quad
\theta>\frac{2}{3}+\frac{4}{3}\eta+\varepsilon.\end{equation}

To compute the off-diagonal $\mathcal{OD}$, we will use
the Poisson summation formula and the stationary phase method.
We expect a main term from the zero frequency and will show that the other terms are lower order terms.

\begin{lem}\label{afterPoisson} Let $\mathcal{OD}$ be defined as in \eqref{offdiagonal2}. Then
\begin{align*}\mathcal{OD}=
-\sqrt{\pi}G  \Im \bigg(&e^{-2\pi i/8} \sum_{c} \sum_{\substack{b_1 \mod{c}, \\ b_2 \mod{c}}}\sum_{\substack{a_1 \mod{c}, \\ a_2 \mod{c}}} \frac{S(a_1(a_1+b_1),a_2(a_2+b_2);c)}{c^{3/2}}e_c(2a_1a_2+a_1b_2+a_2b_1) \cdot \\
& \times \sum_{d_1,d_2}  \sum_{\substack{m_1 \equiv b_1 \mod{c}\\m_2 \equiv b_2 \mod{c}}}\sum_{n_2 \equiv a_2 \mod{c}} \sum_{v\in \mathbb{Z}} \mathcal{I}_v(n_2, m_1, m_2, d_1, d_2, c) \bigg)
\end{align*}
with
\begin{align}\label{integral}
\mathcal{I}_v(n_2, m_1, m_2, d_1, d_2, c):=\int_{-\infty}^\infty &\frac{e_c\big(f(x, n_2,m_1,m_2)-v(x-a_1)\big)}{d_1d_2 \big(x(x+m_1)n_2(n_2+m_2)\big)^{3/4}}\varpi_{c,G} \begin{pmatrix}
d_1 &m_1 &x\\
d_2 &m_2& n_2
\end{pmatrix}dx .
\end{align}
\end{lem}
\begin{proof}
We work with expression \eqref{offdiagonal2} and split the variables $n_1, n_2, m_1, m_2$ into residue classes modulo $c$. The off-diagonal $OD$ is then given by
  \begin{align}\label{Offdiagonalexpression}
 -\sqrt{\pi}G  \Im \bigg(&e^{-2\pi i/8} \sum_{c} \sum_{\substack{b_1 \mod{c}, \\ b_2 \mod{c}}}\sum_{\substack{a_1 \mod{c}, \\ a_2 \mod{c}}} \frac{S(a_1(a_1+b_1),a_2(a_2+b_2);c)}{\sqrt{c}}e_c(2a_1a_2+a_1b_2+a_2b_1) \\
& \times \sum_{d_1,d_2} \mathop{\sum\sum}_{\substack{j\in\{1,2\}\\m_j \equiv b_j \mod{c}\\ n_j \equiv a_j \mod{c}}}\frac{e_c\big(f(n_1,n_2,m_1,m_2)\big)}{d_1d_2 \big(n_1(n_1+m_1)n_2(n_2+m_2)\big)^{3/4}} \cdot \varpi_{c,G} \begin{pmatrix}
d_1 &m_1 &n_1\\
d_2 &m_2& n_2
\end{pmatrix}\bigg). \label{exposum}
 \end{align}
Here we have used the fact that the Kloosterman sum only depends on
the residue classes modulo $c$. The lemma follows now upon
applying the Poisson summation formula (see Lemma \ref{Poisson}) to the summation over $n_1$.
\end{proof}

Next we will analyze the integral
\begin{equation}\label{maintermanderrorterm}
\mathcal{I}_v(n_2, m_1, m_2, d_1, d_2, c)
\end{equation} with the stationary phase method.
We now need to check the conditions of the Lemma \ref{stationaryphase}
so that we can apply it to the quantity $\mathcal{I}_v(n_2, m_1, m_2, d_1, d_2, c)$.
The stationary points $x_v^*(n_2)$ for fixed $v$ are the solutions to the equation
\begin{align*}
\frac{\partial}{\partial x}\big(f-v(x-a_1)\big)&=\sqrt{n_2(n_2+m_2)}\left(\sqrt{\frac{x+m_1}{x}}+\sqrt{\frac{x}{x+m_1}}\right)-2n_2-m_2-v=0,\\
&\frac{\partial^2}{\partial x^2}\big(f-v(x-a_1)\big)=-\frac{m_1^2}{2}\sqrt{\frac{n_2(n_2+m_2)}{(x(x+m_1))^3}}\neq 0.
\end{align*}
Combining with $x\asymp \frac{K}{d_1}$, which is given by the proposition of $\varpi_{c,G}$,
we have
\begin{align*}
\sqrt{n_2(n_2+m_2)}\left(\sqrt{\frac{x+m_1}{x}}+\sqrt{\frac{x}{x+m_1}}\right)-2n_2-m_2= O\left(\frac{m_2^2}{n_2}+\frac{n_2m_1^2}{x^2}\right)=O\left(\frac{1}{d_2}\right).
\end{align*}
For $v\ll \frac{1}{d_2}$, the solutions are given by
\begin{equation}\label{stationarypoint}
x_v^*(n_2)=\frac{m_1}{2}\Big(-1+\frac{v+m_2+2n_2}{\sqrt{(v+m_2)^2+4vn_2}}\Big).
\end{equation}
For $v\gg \frac{1}{d_2}$, there is no stationary point.

 There are two conditions $v=0$ and $v\neq 0$.
 In the case $v\neq 0$ and $v\ll \frac{1}{d_2}$, we have
 \begin{align*}
 \frac{1}{\sqrt{2}}\left(\frac{n_2}{v}\right)^{\frac{1}{2}} \leq\big(1+\frac{4m_2n_2+4n_2^2}{(v+m_2)^2+4vn_2}\big)^{\frac{1}{2}}\leq  6\left(\frac{n_2}{v}\right)^{\frac{1}{2}},
\end{align*}
 which implies
\begin{align*}
3^{-1}m_1\left(\frac{n_2}{v}\right)^{\frac{1}{2}}\leq x_v^*(n_2)\leq 3m_1\left(\frac{n_2}{v}\right)^{\frac{1}{2}}.
\end{align*}
In the case $v=0$, we have
\begin{align*}
 x_0^*(n_2)= \frac{m_1n_2}{m_2}.
\end{align*}

Stationary phase analysis leads to the following estimates:
\begin{align}
\frac{\partial^j}{\partial x^j}f(x,n_2,m_1,m_2)\asymp_j \frac{m_1^2n_2}{x^{j+1}}\quad \text{for} \;j\geq 2
\end{align}
and
\begin{align}
\frac{\partial^j}{\partial x^j}\frac{1}{d_1d_2 \big(x(x+m_1)n_2(n_2+m_2)\big)^{3/4}} \varpi_{c,G} \begin{pmatrix}
d_1 &m_1 &x\\
d_2 &m_2& n_2
\end{pmatrix}\ll \frac{K}{d_1d_2}\cdot \left(\frac{1}{n_2x}\right)^{\frac{3}{2}}\left(\frac{cK^2}{n_2 x^2}\right)^j,
\end{align}
which is the conclusion of following estimates
\begin{align*}
&\frac{\partial^k}{\partial x^k}\psi_1\Big(\frac{\sqrt{u}G+K}{2\pi d_1(2x+m_1)}\Big)\ll \frac{1}{x^k},&
\frac{\partial^k}{\partial x^k}\exp\Big(- \frac{(\sqrt{u}G+K) m_1^2}{(2x+m_1)^2}\Big)\ll \left(\frac{K^{\varepsilon}}{x}\right)^k,\\
&\frac{\partial^k}{\partial x^k}e^{ic \frac{uG^2+2\sqrt{u}KG+K^2}{8\pi (n_1(n_1+m_1)n_2(n_2+m_2))^{1/2}}}\ll \left(\frac{cK^2}{n_2 x^2}\right)^k.
\end{align*}
\subsection{Estimates for $v=0$}
 In the case $v=0$, without loss of generality,
 we suppose $\beta^{-1}\frac{K}{d_2}\leq n_2\leq \beta \frac{K}{d_2}$.
 More precisely, $\beta$ is only relative to the supported set of $\psi_2$.
 Then \eqref{stationarypoint} with $v=0$ leads to an upper bound and a lower bound naturally,
$$
\beta^{-1}\frac{Km_1}{d_2m_2}\leq x_0^*(n_2)= \frac{m_1n_2}{m_2}\leq \beta \frac{Km_1}{d_2m_2}.
$$

In the case $x_0^*(n_2)\in \left(\alpha^{-1}\frac{K}{d_1},\alpha\frac{K}{d_1}\right)$,
where this interval is the supported set of variable $n_1$ and $\alpha$ is
only relative to the supported set of $\psi_1$.
We then have the restriction
$(\alpha \beta)^{-1}\frac{K}{d_1}<\frac{Km_1}{d_2m_2}< \alpha \beta\frac{K}{d_1}$.
To apply the stationary phase method,
we take $x\asymp \frac{Km_1}{d_2m_2}$.
Note that
\begin{align*}
\frac{\partial^j}{\partial x^j}f(x,n_2,m_1,m_2)\asymp_j m_1n_2\left(\frac{Km_1}{d_2m_2}\right)^{-j}\coloneqq YQ^{-j}\quad \text{for} \;j\geq 2,
\end{align*}
\begin{align*}
\frac{\partial^j}{\partial x^j}\frac{1}{d_1d_2 \big(x(x+m_1)n_2(n_2+m_2)\big)^{3/4}}\varpi_{c,G} &  \begin{pmatrix}
d_1 &m_1 &x\\
d_2 &m_2& n_2
\end{pmatrix}\ll\frac{K}{d_1d_2}\cdot \left(\frac{d_1}{n_2K}\right)^{\frac{3}{2}}\left(\frac{d_1K^{\varepsilon}}{G}\right)^j\\
&\ll \frac{(d_1d_2)^{\frac{1}{2}}}{K^2}\left(\frac{d_1K^{\varepsilon}}{G}\right)^j
\coloneqq XV^{-j},
\end{align*}
and the parameters are taken as
\begin{align*}
X&=\frac{(d_1d_2)^{\frac{1}{2}}}{K^2},&V=&\frac{G^{1-\varepsilon}}{d_1},\\
Y&=m_1n_2c^{-1},&Q=&\frac{Km_1}{d_2m_2}.\nonumber
\end{align*}
Also note that the length of supported set
is $V_1=\frac{K}{d_1}$,  $Z=Q+X+Y+V_1+1\ll \frac{K}{d_1}$.
We need to have the following two inequalities
$$\frac{m_1K}{d_2} \geq \left(\frac{K}{d_1}\right)^{3/20},\quad \frac{K}{d_1} \geq \frac{K^{\theta-\varepsilon}}{d_1}\geq\left(\frac{K}{d_1}\right)^{41/40}\left(\frac{d_2m_2}{Km_1}\right)^{1/2},$$
The first inequality holds when
\begin{align}
\eta-\frac{1}{2}\delta\geq \frac{3}{40} .
\end{align}
The second inequality holds when
\begin{align}
\theta+\frac{1}{2}\eta-\delta\geq \frac{31}{40} +\varepsilon.
\end{align}
Now by Lemma \ref{stationaryphase} we conclude that
\begin{align}\label{errortermfromM}
\mathcal{I}_0(n_2, m_1, m_2, d_1, d_2, c)=&\frac{e^{-\frac{i \pi}{4}}\sqrt{2c}}{d_1d_2m_1n_2^2}  \varpi_{c,G} \begin{pmatrix}
d_1 &m_1 &x_0^{*}(n_2)\\
d_2 &m_2& n_2
\end{pmatrix}+  O\Big(\frac{c^{\frac{3}{2}}d_1^{\frac{5}{2}}d_2^{\frac{1}{2}}}{G^2K^{2-\varepsilon}m_2^{\frac{3}{2}}} \Big)\\
&+O\Big(\frac{c^{\frac{5}{6}}d_1^{\frac{1}{2}}d_2^{\frac{1}{2}}m_2^{\frac{1}{2}}}{K^{2}(m_1n_2)^{\frac{4}{3}}} \Big)+O\Big(\frac{\sqrt{c} d_2}{d_1m_1K^{\frac{3}{2}}} \Big).\nonumber
\end{align}

The contribution from these error terms to $\mathcal{OD}$ are
\begin{align}\label{oderrorterm2}
G\sum_{c}\sum_{\substack{b_1 \mod{c}, \\ b_2 \mod{c}}}\sum_{\substack{a_1 \mod{c}, \\ a_2 \mod{c}}}\frac{|S(a_1(a_1+b_1),a_2(a_2+b_2);c)|}{c^{3/2}}\sum_{d_1,d_2\atop m_1,m_2 }\sum_{n_2}\frac{c^{\frac{3}{2}}d_1^{\frac{5}{2}}d_2^{\frac{1}{2}}}{G^2K^{2-\varepsilon}m_2^{\frac{3}{2}}} \ll K^{\frac{9}{2}-\frac{1}{2}\eta+3\delta+\varepsilon}G^{-6},
\end{align}
\begin{align}\label{oderrorterm3}
G\sum_{c}\sum_{\substack{b_1 \mod{c}, \\ b_2 \mod{c}}}\sum_{\substack{a_1 \mod{c}, \\ a_2 \mod{c}}}\frac{|S(a_1(a_1+b_1),a_2(a_2+b_2);c)|}{c^{3/2}}\sum_{d_1,d_2\atop m_1,m_2}\sum_{n_2}\frac{c^{\frac{5}{6}}d_1^{\frac{1}{2}}d_2^{\frac{1}{2}}m_2^{\frac{1}{2}}}{K^{2}(m_1n_2)^{\frac{4}{3}}}\ll K^{\frac{11}{4}+\frac{11}{6}\delta-\frac{1}{3}\eta+\varepsilon}G^{-\frac{10}{3}},
\end{align}
\begin{align}\label{oderrorterm4}
G\sum_{c}\sum_{\substack{b_1 \mod{c}, \\ b_2 \mod{c}}}\sum_{\substack{a_1 \mod{c}, \\ a_2 \mod{c}}}\frac{|S(a_1(a_1+b_1),a_2(a_2+b_2);c)|}{c^{3/2}}\sum_{d_1,d_2\atop m_1,m_2}\sum_{n_2}\frac{\sqrt{c} d_2}{d_1m_1K^{\frac{3}{2}}}\ll K^{3+\varepsilon}G^{-2},
\end{align}
and the following conditions ensure that above bounds  are less than $K^{\frac{1}{2}}G$,
\begin{align}
&\theta>\frac{4}{7}-\frac{1}{14}\eta+\frac{3}{14}\delta,\\
&\theta>\frac{27}{52}-\frac{1}{13}\eta+\frac{11}{26}\delta,\\
&\theta>\frac{5}{6}.
\end{align}

	In the case $x_0^*(n_2)\notin \left(\alpha^{-1}\frac{K}{d_1},\alpha\frac{K}{d_1}\right)$,
which implies
$(\alpha \beta)^{-1}\frac{K}{d_1}\geq \frac{Km_1}{d_2m_2}$ or $\frac{Km_1}{d_2m_2}\geq
\alpha \beta\frac{K}{d_1}$.
We need to check a lower bound of $f' (x,n_2,m_1,m_2)$.
Combining with the monotonicity of $f'$, if  $x_0^*(n_2)$ is on the left side of $\left(\alpha^{-1}\frac{K}{d_1},\alpha\frac{K}{d_1}\right)$, we have
\begin{align}
\frac{\partial}{\partial x}f(x,n_2,m_1,m_2)&\Big|_{x=\beta \frac{Km_1}{d_2m_2}}=\Big[\sqrt{n_2(n_2+m_2)}\big(\sqrt{\frac{x+m_1}{x}}+\sqrt{\frac{x}{x+m_1}}\big)-2n_2-m_2\Big]\Big|_{x=\beta \frac{Km_1}{d_2m_2}}\\
=&n_2\Big(1+\frac{1}{2}\frac{m_2}{n_2}-\frac{1}{8}\frac{m_2^2}{n_2^2}+O\big(\frac{m_2^3}{n_2^3}\big)\Big)\Big(2+\frac{1}{4}\frac{d_2^2	m_2^2}{\beta^2 K^2}+O\big(\frac{d_2^3m_2^3}{K^3}\big)\Big)-2n_2-m_2\nonumber\\
=&\frac{1}{4}\frac{d_2^2 m_2^2n_2}{\beta^2 K^2}-\frac{1}{4}\frac{m_2^2}{n_2}+O\Big(\frac{m_2^3}{n_2^2} \Big),\nonumber\\
=&-\frac{m_1^2}{4}\Big(\frac{1}{n_2}-\frac{d_2}{\beta K}\Big)+O\Big(\frac{m_2^3}{n_2^2} \Big),\nonumber
\end{align}
by the fact that  $\psi_2$ is a compact supported function, we say $$ \frac{\partial}{\partial x}f(x,n_2,m_1,m_2)\Big|_{x=\beta \frac{Km_1}{d_2m_2}}\gg \frac{m_2^2}{n_2}.$$ And  if  $x_0^*(n_2)$  is on the right side of $\left(\alpha^{-1}\frac{K}{d_1},\alpha\frac{K}{d_1}\right)$, we have the same lower bound
$$ \frac{\partial}{\partial x}f(x,n_2,m_1,m_2)\Big|_{x=\beta^{-1} \frac{Km_1}{d_2m_2}}\gg \frac{m_2^2}{n_2}.$$

Set
\begin{align}
X=\frac{(d_1d_2)^{\frac{1}{2}}}{K^2}, \quad U=\frac{G^{1-\varepsilon}}{d_1}, \quad R=\frac{m_2^2}{n_2},
\quad Y=\frac{m_1^2n_2d_1}{cK}, \quad Q=\frac{K}{d_1}.\nonumber
\end{align}
If $(\alpha \beta)^{-1}d_2m_2\geq d_1m_1$ ,  by integrating by parts,  we have
\begin{align}\label{errortermfrom0}
I &\ll \frac{K}{d_1}X\left((UR)^{-1}+Y(QR)^{-2}\right)=\frac{d_1^{\frac{1}{2}}}{G^{1-\varepsilon}d_2^{\frac{1}{2}}m_2^{2}}+\frac{d_1^{\frac{5}{2}}m_1^2}{cKd_2^{\frac{5}{2}}m_2^4}\\
&\ll \frac{d_1^{\frac{1}{2}}}{G^{1-\varepsilon}d_2^{\frac{1}{2}}m_2^{2}}+\frac{d_1^{\frac{1}{2}}}{cKd_2^{\frac{1}{2}}m_2^2}.\nonumber
\end{align}

The first error term contributes to $\mathcal{OD}$ at most
\begin{align}\label{oderrorterm5}
G\sum_{c}\sum_{\substack{b_1 \mod{c}, \\ b_2 \mod{c}}}\sum_{\substack{a_1 \mod{c}, \\ a_2 \mod{c}}}\frac{S(a_1(a_1+b_1),a_2(a_2+b_2);c)}{c^{3/2}}\mathop{\sum\sum}_{(\alpha \beta)^{-1}d_2m_2\geq d_1m_1}\sum_{n_2}\frac{d_1^{\frac{1}{2}}}{G^{1-\varepsilon}d_2^{\frac{1}{2}}m_2^{2}}\ll K^{\frac{7}{2}+\delta}G^{-\frac{5}{2}},
\end{align}
which requires
\begin{align}
\theta>\frac{6}{7}+\frac{2}{7}\delta.
\end{align}
The second error term contributes to $\mathcal{OD}$ at most
\begin{align}\label{oderrorterm6}
G\sum_{c}\sum_{\substack{b_1 \mod{c}, \\ b_2 \mod{c}}}\sum_{\substack{a_1 \mod{c}, \\ a_2 \mod{c}}}\frac{S(a_1(a_1+b_1),a_2(a_2+b_2);c)}{c^{3/2}}\mathop{\sum\sum}_{(\alpha \beta)^{-1}d_2m_2\geq d_1m_1}\sum_{n_2}\frac{d_1^{\frac{1}{2}}}{cKd_2^{\frac{1}{2}}m_2^2}\ll K^{\frac{3}{2}+\delta}G^{-\frac{1}{2}},
\end{align}
which requires
\begin{align}
\theta>\frac{2}{3}+\frac{2}{3}\delta.
\end{align}
\begin{remark}\label{dm1lessthandm2}
The last two terms are related to  the sums over $d_1m_1\leq (\alpha \beta)^{-1}d_2m_2, n_2$
and $v=0$,  and we have prove that they contribute error terms.
But the other case, which is the sum over $d_2m_2\leq (\alpha \beta)^{-1}d_1m_1, n_2$ and $v=0$,
can not deduce an error term less then $K^{\frac{1}{2}}G$ by same operation.
If we can further prove that the sum over $d_1m_1\leq (\alpha \beta)^{-1}d_2m_2, n_2$ and
all $v$ contributes error terms, the sums
over $d_2m_2\leq (\alpha \beta)^{-1}d_1m_1, n_2$ and $v$ will have same upper bounds
like \eqref{oderrorterm5}and \eqref{oderrorterm6} for symmetric.
\end{remark}

\subsection{Estimates for $v\neq 0$}

As in Section 2.8, there are two cases when $v\ll d_2^{-1}$.
One is  $x_v^*(n_2)\in \left(\alpha^{-1}\frac{K}{d_1},\alpha\frac{K}{d_1}\right)$,
and the other one is $x_v^*(n_2)\notin \left(\alpha^{-1}\frac{K}{d_1},\alpha\frac{K}{d_1}\right)$.

In the first case $x_v^*(n_2)\in \left(\alpha^{-1}\frac{K}{d_1},\alpha\frac{K}{d_1}\right)$, by choosing
\begin{align*}
X&=\frac{(d_1d_2)^{\frac{1}{2}}}{K^2},&V=&\frac{G^{1-\varepsilon}}{d_1},\\
Y&=\frac{d_1m_1^2n_2}{cK} ,&Q=&\frac{K}{d_1},\nonumber
\end{align*}
we have
\begin{align}\label{errortermfromv}\mathcal{I}_v(n_2, m_1, m_2, d_1, d_2, c)=&\frac{e_c\Big(a_1v+\frac{m_1}{2}\big(v+m_2-\sqrt{(v+m_2)^2+4vn_2}\big)\Big)e^{-\pi i/4}\sqrt{c}}{d_1d_2\sqrt{|f''(x_v^*(n_2))|}\cdot \big(x_v^{*}(n_2)(x_v^*(n_2)+m_1)n_2(n_2+m_2)\big)^{3/4}}\\
&\times\varpi_{c,G} \begin{pmatrix}
d_1 &m_1 &x_v^{*}(n_2)\\
d_2 &m_2& n_2
\end{pmatrix}+ O\bigg(\frac{c^{\frac{3}{2}}d_2^2}{d_1^{\frac{1}{2}}G^{2}K^{\frac{1}{2}-\varepsilon}m_1^3} \bigg)+ O\bigg(\frac{c^{\frac{5}{6}}d_1^{\frac{1}{6}}d_2^{\frac{4}{3}}}{K^{\frac{5}{2}-\varepsilon}m_1^{\frac{5}{3}}} \bigg)\nonumber\\
\ll&\frac{\sqrt{c}K}{d_1d_2m_1n_2^{2}}+\frac{c^{\frac{3}{2}}d_2^2}{d_1^{\frac{1}{2}}G^{2}K^{\frac{1}{2}-\varepsilon}m_1^3} + \frac{c^{\frac{5}{6}}d_1^{\frac{1}{6}}d_2^{\frac{4}{3}}}{K^{\frac{5}{2}-\varepsilon}m_1^{\frac{5}{3}}} .\nonumber
\end{align}
Note that $m_1\big(\frac{n_2}{v}\big)^{\frac{1}{2}}\asymp \frac{K}{d_1}$.
The contributions to $\mathcal{OD}$ are
\begin{align}\label{oderrorterm7}
G\sum_{c}\sum_{\substack{b_1 \mod{c}, \\ b_2 \mod{c}}}\sum_{\substack{a_1 \mod{c}, \\ a_2 \mod{c}}}\frac{|S(a_1(a_1+b_1),a_2(a_2+b_2);c)|}{c^{3/2}}\sum_{d_1,d_2\atop m_1,m_2}\sum_{n_2,v}\frac{\sqrt{c}K}{d_1d_2m_1n_2^{2}}\ll K^{4+\delta+\varepsilon}G^{-3},
\end{align}
\begin{align}\label{oderrorterm8}
G\sum_{c}\sum_{\substack{b_1 \mod{c}, \\ b_2 \mod{c}}}\sum_{\substack{a_1 \mod{c}, \\ a_2 \mod{c}}}\frac{|S(a_1(a_1+b_1),a_2(a_2+b_2);c)|}{c^{3/2}}\sum_{d_1,d_2\atop m_1,m_2}\sum_{n_2}\frac{c^{\frac{3}{2}}d_2^2}{d_1^{\frac{1}{2}}G^{2}K^{\frac{1}{2}-\varepsilon}m_1^3} \ll K^{4+\frac{5}{2}\delta+\varepsilon}G^{-5},
\end{align}
\begin{align}\label{oderrorterm9}
G\sum_{c}\sum_{\substack{b_1 \mod{c}, \\ b_2 \mod{c}}}\sum_{\substack{a_1 \mod{c}, \\ a_2 \mod{c}}}\frac{S(a_1(a_1+b_1),a_2(a_2+b_2);c)}{c^{3/2}}\sum_{d_1,d_2\atop m_1,m_2}\sum_{n_2}\frac{c^{\frac{5}{6}}d_1^{\frac{1}{6}}d_2^{\frac{4}{3}}}{K^{\frac{5}{2}-\varepsilon}m_1^{\frac{5}{3}}}\ll K^{2+\frac{11}{6}\delta+\varepsilon}G^{-\frac{4}{3}}
\end{align}
which to be less than $K^{\frac{1}{2}}G$, require the following conditions:
\begin{align}
&\theta>\frac{7}{8}+\frac{1}{4}\delta,\\
&\theta>\frac{7}{12}+\frac{5}{12}\delta,\\
&\theta>\frac{9}{14}+\frac{11}{14}\delta.
\end{align}

In the second case $x_v^*(n_2)\notin \left(\alpha^{-1}\frac{K}{d_1},\alpha\frac{K}{d_1}\right)$,
we have $3m_1\left(\frac{n_2}{v}\right)^{\frac{1}{2}}\leq \beta^{-1}\frac{K}{d_1}$
or  $3^{-1}m_1\left(\frac{n_2}{v}\right)^{\frac{1}{2}}\geq \beta\frac{K}{d_1}$. Therefore we have
\begin{align}
\frac{\partial}{\partial x}f(x,n_2,m_1,m_2)&\Big|_{x=3m_1\left(\frac{n_2}{v}\right)^{\frac{1}{2}}}=\Big[\sqrt{n_2(n_2+m_2)}\big(\sqrt{\frac{x+m_1}{x}}+\sqrt{\frac{x}{x+m_1}}\big)-2n_2-m_2-v\Big]\Big|_{x=3m_1\left(\frac{n_2}{v}\right)^{\frac{1}{2}}}\\
=&n_2\Big(1+\frac{1}{2}\frac{m_2}{n_2}-\frac{1}{8}\frac{m_2^2}{n_2^2}+O\big(\frac{m_2^3}{n_2^3}\big)\Big)\Big(2+\frac{1}{36}\frac{v}{n_2}+O\Big(\big(\frac{v}{n_2}\big)^{\frac{3}{2}}\Big)\Big)-2n_2-m_2-v\nonumber\\
=&-\frac{35}{36}v-\frac{m_2^2}{n_2}+O\Big(\frac{m_2v}{n_2} \Big)\nonumber\\
\gg& v\nonumber
\end{align}
and similarly
\begin{align}
\frac{\partial}{\partial x}f(x,n_2,m_1,m_2)&\Big|_{x=3^{-1}m_1\left(\frac{n_2}{v}\right)^{\frac{1}{2}}}
\gg v.
\end{align}
Applying Lemma \ref{integrationbyparts} with
\begin{align}
X=\frac{G}{(n_1n_2)^{\frac{1}{2}}K^2}, \quad U=\frac{K^{2\theta-1-\varepsilon}}{d_1}, \quad R=v,\quad
Y=\frac{m_1^2n_2}{cn_1}	, \quad Q=\frac{K}{d_1},\nonumber
\end{align}
we have
\begin{align}\label{oderrorterm10}
I \ll_A \frac{K}{d_1} \frac{G}{(n_1n_2)^{\frac{1}{2}}K^2} \left[\left(v^2\frac{cK}{d_1} \right)^{-\frac{A}{2}} + \left(v\frac{K^{2\theta-1-\varepsilon}}{d_1}\right)^{-A}\right],
\end{align}
which is negligible, since we have restricted $\theta>\frac{3}{4}+\eta+\varepsilon$.
When $v\gg d_2^{-1}$, the upper bound is also negligible,
because $\frac{\partial}{\partial x}\left(f(x,n_2,m_1,m_2)-v(x-a_1)\right)\gg v.$

Combining with Remark \ref{dm1lessthandm2}, we see that
the sum over $d_2m_2\leq(\alpha \beta)^{-1}d_1m_1, n_1, n_2$ is also less
than $K^{\frac{1}{2}}G.$ Therefore, we have proved the following.
\begin{lem}\label{maintermlemma}
Let $\mathcal{I}_v(n_2, m_1, m_2, d_1, d_2, c)$ be defined as in \eqref{integral}. Then
\begin{align}
\sum_{n_2 \equiv a_2\mod{c}} \sum_{v}\mathcal{I}_v(n_1, m_1, m_2, d_1, d_2, c)
= \sum_{n_2\equiv a_2\mod{c}} \frac{e^{-\frac{i \pi}{4}}\sqrt{2c}}{d_1d_2m_1n_2^2}
\varpi_{c,G}& \begin{pmatrix}
d_1 &m_1 &x_0^{*}(n_2)\\
d_2 &m_2& n_2
\end{pmatrix}\\&+\text{lower order terms},\nonumber
\end{align}
where the contribution from the lower order terms to $\mathcal{OD}$
is less than $K^{\frac{1}{2}}G$ with suitable $\theta, \eta, \delta.$
Moreover, if we choose $\theta=\frac{9}{10},\delta=\frac{9}{160}, \eta=\frac{13}{80}$,
then the contribution from the lower order terms
to $\mathcal{OD}$ will be $O(K^{\frac{11}{8}})$.
\end{lem}
\subsection{Possion Summation}
Applying Possion summation fomula to the variable $n_2$, we have
\begin{align}\label{possionovern2}
\sum_{n_2\equiv a_2\mod{c}} \frac{1}{d_1d_2m_1n_2^2}  \varpi_{c,G}& \begin{pmatrix}
d_1 &m_1 &x_0^{*}(n_2)\\
d_2 &m_2& n_2
\end{pmatrix}
\\
&=\frac{1}{c}\sum_{u}\int_{-\infty}^{\infty}\frac{e_c\big(-u(x-a_2)\big)}{d_1d_2m_1y^2}  \varpi_{c,G} \begin{pmatrix}
d_1 &m_1 &x_0^{*}(y)\\
d_2 &m_2& y
\end{pmatrix}dy.\nonumber
\end{align}

We care about the upper bound of each derivative of $\varpi_{c,G}(y)$,
noting the restriction $d_1m_1\asymp  d_2m_2$ and $y\asymp \frac{K}{d_2}$,
\begin{align*}
&\frac{\partial^k}{\partial y^k}\psi_1\Big(\frac{(\sqrt{u}G+K)m_2}{2\pi d_1m_1(2y+m_2)}\Big)\ll \frac{1}{y^k},
&\frac{\partial^k}{\partial y^k}&\psi_2\Big(\frac{\sqrt{u}G+K}{2\pi d_2(2y+m_2)}\Big)\ll \frac{1}{y^k},\\
&\frac{\partial^k}{\partial y^k}\exp\Big(- \frac{2(\sqrt{u}G+K) m_2^2}{(2y+m_2)^2}\Big)\ll\left(\frac{K^{\varepsilon}}{y}\right)^k,
&\frac{\partial^k}{\partial y^k}&e^{ic \frac{(uG^2+2\sqrt{u}KG+K^2)m_2}{8\pi m_1y(y+m_2)}}\ll \left(\frac{cK^{2}m_2}{ m_1y^3}\right)^k.
\end{align*}
Therefore,
\begin{align*}
\frac{\partial^k}{\partial y^k}\frac{1}{d_1d_2m_1y^2}  \varpi_{c,G} \begin{pmatrix}
d_1 &m_1 &x_0^{*}(y)\\
d_2 &m_2& y
\end{pmatrix}\ll \frac{K}{d_1d_2m_1}\left(\frac{d_2}{ G}\right)^k\ll K\cdot K^{(\delta-\theta)k}.
\end{align*}
Since $\theta>\frac{1}{2}$ and $\delta<\frac{1}{2}$, we see that the terms in \eqref{possionovern2} corresponding to $u\neq 0$ is negligible by integration by parts. We transform the sum over $n_2$ to
\begin{align}\label{possionovern2}
\sum_{n_2\equiv a_2\mod{c}} \frac{1}{d_1d_2m_1n_2^2}  \varpi_{c,G} \begin{pmatrix}
d_1 &m_1 &x_0^{*}(n_2)\\
d_2 &m_2& n_2
\end{pmatrix}
&=\frac{1}{c}\int_{-\infty}^{\infty}\frac{1}{d_1d_2m_1y^2}  \varpi_{c,G} \begin{pmatrix}
d_1 &m_1 &x_0^{*}(y)\\
d_2 &m_2& y
\end{pmatrix}dy+O(K^{-A})\\
&=\frac{1}{c}\sum_{n_2\in \mathbb{N}} \frac{1}{d_1d_2m_1n_2^2}  \varpi_{c,G} \begin{pmatrix}
d_1 &m_1 &x_0^{*}(n_2)\\
d_2 &m_2& n_2
\end{pmatrix}+O(K^{-A}).\nonumber
\end{align}

	We do the same work for $m_1$ and $m_2$,
and the key step is to find the upper bounds of each derivative
of $\varpi_{c,G}(m_1)$ and $\varpi_{c,G}(m_1)$. Recall that
\begin{align*}
\varpi_{c,G}\begin{pmatrix}
d_1 &m_1 &x_0^{*}(n_2)\\
d_2 &m_2& n_2
\end{pmatrix}=&\frac{K}{16}\int_0^\infty \frac{h(\sqrt{u})}{\sqrt{2\pi u}}\Big(1+\frac{G\sqrt{u}}{K}\Big) \psi_1\Big(\frac{(\sqrt{u}G+K)m_2}{2\pi d_1m_1(2n_2+m_2)}\Big)\psi_2\Big(\frac{(\sqrt{u}G+K)}{2\pi d_2(2n_2+m_2)}\Big) \\
 &\times  \exp\Big(- \frac{2(\sqrt{u}G+K) m_2^2}{(2n_2+m_2)^2}\Big)  e^{ic \frac{(uG^2+2\sqrt{u}KG+K^2)m_2}{8\pi m_1n_2(n_2+m_2)}} du,
\end{align*}
which implies the restriction $d_1m_1\asymp d_2m_2$. We have
\begin{align*}
\frac{\partial^k}{\partial y^k}\frac{1}{d_1d_2yn_2^2}  \varpi_{c,G} \begin{pmatrix}
d_1 &y &x_0^{*}(n_2)\\
d_2 &m_2& n_2
\end{pmatrix}\ll \frac{K}{d_1d_2n_2^2}\left(\frac{d_1K^{1+\varepsilon}}{ d_2m_2G}\right)^k
\ll K\cdot K^{(1+\varepsilon+\delta-\theta-\eta)k},
\end{align*}
where we have used the following estimates
\begin{align*}
&\frac{\partial^k}{\partial y^k}\psi_1\Big(\frac{(\sqrt{u}G+K)m_2}{2\pi d_1y(2n_2+m_2)}\Big)\ll \frac{K^{\varepsilon}}{y^k},
&\frac{\partial^k}{\partial y^k}e^{ic \frac{(uG^2+2\sqrt{u}KG+K^2)m_2}{8\pi yn_2(n_2+m_2)}}
\ll \left(\frac{cK^{2}m_2}{ y^2n_2^{2}}\right)^k&.
\end{align*}
If we further suppose
\begin{align}
\theta+\eta-\delta>1,
\end{align}
then we have
\begin{align}
\sum_{m_1\equiv b_1\mod{c}} \frac{1}{d_1d_2m_1n_2^2}  \varpi_{c,G} \begin{pmatrix}
d_1 &m_1 &x_0^{*}(n_2)\\
d_2 &m_2& n_2
\end{pmatrix}
=\frac{1}{c}\sum_{m_1\in \mathbb{N}} \frac{1}{d_1d_2m_1n_2^2}  \varpi_{c,G} &\begin{pmatrix}
d_1 &m_1 &x_0^{*}(n_2)\\
d_2 &m_2& n_2
\end{pmatrix}\\&+O(K^{-A}).\nonumber
\end{align}

For the sum over $m_2$, combining with $d_2y\asymp d_1m_1$, we have
\begin{align*}
\frac{\partial^k}{\partial y^k}\frac{1}{d_1d_2m_1n_2^2}  \varpi_{c,G} \begin{pmatrix}
d_1 &m_1 &x_0^{*}(n_2)\\
d_2 &y& n_2
\end{pmatrix}\ll \frac{K}{d_1d_2m_1n_2^2}\left(\frac{d_2K^{1+\varepsilon}}{d_1m_1G}\right)^k\ll K\cdot K^{(1+\varepsilon+\delta-\theta-\eta)k}
\end{align*}
where we have used the following estimates
\begin{align*}
&\frac{\partial^k}{\partial y^k}\psi_1\Big(\frac{(\sqrt{u}G+K)y}{2\pi d_1m_1(2n_2+y)}\Big)\ll \frac{1}{y^k},
&\frac{\partial^k}{\partial y^k}&\psi_2\Big(\frac{\sqrt{u}G+K}{2\pi d_2(2n_2+y)}\Big)\ll \frac{1}{y^k},\\
&\frac{\partial^k}{\partial y^k}\exp\Big(- \frac{2(\sqrt{u}G+K) y^2}{(2n_2+y)^2}\Big)\ll \frac{K^{\varepsilon}}{y^k},
&\frac{\partial^k}{\partial y^k}&e^{ic \frac{(uG^2+2\sqrt{u}KG+K^2)y}{8\pi m_1n_2(n_2+y)}}
\ll \left(\frac{cK^{2}}{ m_1n_2(n_2+y)}\right)^k.&
\end{align*}
For same reason, one has
\begin{align}
\sum_{m_2\equiv b_2\mod{c}} \frac{1}{d_1d_2m_1n_2^2}  \varpi_{c,G} \begin{pmatrix}
d_1 &m_1 &x_0^{*}(n_2)\\
d_2 &m_2& n_2
\end{pmatrix}
=\frac{1}{c}\sum_{m_2\in \mathbb{N}} \frac{1}{d_1d_2m_1n_2^2}  \varpi_{c,G} &\begin{pmatrix}
d_1 &m_1 &x_0^{*}(n_2)\\
d_2 &m_2& n_2
\end{pmatrix}\\&+O(K^{-A}).\nonumber
\end{align}

\subsection{Evaluating the Off-Diagonal terms}
Now the expression of $\mathcal{OD}$ becomes
\begin{align}\label{offdiafinal}
\mathcal{OD}=
-\sqrt{2\pi}G  \Im \bigg(&e^{-\frac{i\pi}{2}} \sum_{c} \sum_{\substack{b_1 \mod{c}, \\ b_2 \mod{c}}}\sum_{\substack{a_1 \mod{c}, \\ a_2 \mod{c}}} \frac{S(a_1(a_1+b_1),a_2(a_2+b_2);c)}{c^{4}}e_c(2a_1a_2+a_1b_2+a_2b_1)\\
& \times \sum_{d_1,d_2}  \sum_{m_1,m_2}\sum_{n_2}  \frac{1}{d_1d_2m_1n_2^2}  \varpi_{c,G} \begin{pmatrix}
d_1 &m_1 &x_0^{*}(n_2)\\
d_2 &m_2& n_2
\end{pmatrix} \bigg)+\text{error terms}\nonumber\\
=-\sqrt{2\pi}G  \Im \bigg(&e^{-\frac{i\pi}{2}} \sum_{c} \frac{\phi(c)}{c}
\mathop{\sum\sum\sum}_{d_1,d_2,n_2\in \mathbb{N}^+\atop m_1,m_2\in\mathbb{N}/\{0\}}
\frac{1}{d_1d_2m_1n_2^2}  \varpi_{c,G} \begin{pmatrix}
d_1 &m_1 &x_0^{*}(n_2)\\
d_2 &m_2& n_2
\end{pmatrix} \bigg)+\text{error terms}.\nonumber
\end{align}
By the Taylor expansion, we have
\begin{align}\label{gfinal}
\varpi_{c,G}\begin{pmatrix}
d_1 &m_1 &x_0^{*}(n_2)\\
d_2 &m_2& n_2
\end{pmatrix}=&\frac{K}{16}\int_0^\infty \frac{h(\sqrt{u})}{\sqrt{2\pi u}}\Big(1+\frac{G\sqrt{u}}{K}\Big) \psi_1\Big(\frac{(\sqrt{u}G+K)m_2}{2\pi d_1m_1(2n_2+m_2)}\Big)\psi_2\Big(\frac{(\sqrt{u}G+K)}{2\pi d_2(2n_2+m_2)}\Big) \\
 &\times  \exp\Big(- \frac{2(\sqrt{u}G+K) m_2^2}{(2n_2+m_2)^2}\Big)  e^{ic \frac{(uG^2+2\sqrt{u}KG+K^2)m_2}{8\pi m_1n_2(n_2+m_2)}} du\nonumber\\
 =&\frac{K}{16}\int_0^\infty \frac{h(\sqrt{u})}{\sqrt{2\pi u}}\Big(1+\frac{G\sqrt{u}}{K}\Big) \psi_1\Big(\frac{(\sqrt{u}G+K)m_2}{4\pi d_1m_1n_2}\Big)\psi_2\Big(\frac{(\sqrt{u}G+K)}{4\pi d_2n_2}\Big) \nonumber\\
 &\times  \exp\Big(- \frac{(\sqrt{u}G+K) m_2^2}{2n_2^2}\Big)  e^{ic \frac{(uG^2+2\sqrt{u}KG+K^2)m_2}{8\pi m_1n_2^2}} du+O(K^{\frac{3}{2}+\varepsilon}G^{-1})\nonumber\\
 =& \frac{G}{16}\int_0^\infty \frac{g(\sqrt{u})\sqrt{u}}{\sqrt{2\pi u}} \psi_1\Big(\frac{\sqrt{u}Gm_2}{4\pi d_1m_1n_2}\Big)\psi_2\Big(\frac{\sqrt{u}G}{4\pi d_2n_2}\Big)\exp\Big(- \frac{\sqrt{u}G m_2^2}{2n_2^2}\Big)  e^{ic \frac{uG^2m_2}{8\pi m_1n_2^2}} du\nonumber\\
 &+O(K^{\frac{3}{2}+\varepsilon}G^{-1})\nonumber,
 \end{align}
which contributes to $\mathcal{OD}$ by an error term at most $K^{2}G^{-1}.$

\begin{lem}\label{offdiagonalfinaloutcome} We have
\begin{align}
\mathcal{OD}=O(K^{\frac{11}{8}}).\nonumber
\end{align}

\end{lem}
\begin{proof}
From \eqref{offdiafinal} and \eqref{gfinal}, by ignoring the error terms, we focus on
\begin{align*}
\mathcal{OD'}=&\frac{- \sqrt{2\pi} G^2}{16}  \sum_{c} \frac{\phi(c)}{c}  \sum_{d_1,d_2} \sum_{n_2} \sum_{m_1, m_2} \frac{1}{d_1d_2m_1n_2^2} \cdot \\
&\times\int_0^\infty \frac{g(\sqrt{u})\sqrt{u} }{\sqrt{2\pi u}}\psi_1\Big(\frac{\sqrt{u}Gm_2}{4\pi m_1 d_1 n_2}\Big)\psi_2\Big(\frac{\sqrt{u}G}{4\pi d_2n_2}\Big)\exp\Big(-\frac{\sqrt{u}Gm_2^2}{2n_2^2}\Big)\cos\Big( u \frac{cG^2m_2}{8\pi m_1n_2^2}\Big)du .\nonumber
\end{align*}
To evaluate this expression asymptotically we perform an inverse Mellin transform on $\psi_1, \psi_2$ and the exponential function.
Then $\mathcal{OD}$ is equal to
\begin{align}
-\frac{\sqrt{2\pi} G^2}{16}& \sum_{c} \frac{\phi(c)}{c}  \sum_{d_1,d_2} \sum_{n_2} \sum_{m_1, m_2} \frac{1}{d_1d_2m_1n_2^2} \cdot \\
 &\times \frac{1}{(2\pi i)^3} \int_{(1/2+\varepsilon)}\int_{(2)}\int_{(1+\varepsilon)} \tilde{\psi_1}(s_1)\tilde{\psi_2}(s_2)\Gamma(s_3) \Big(\frac{Gm_2}{4\pi m_1d_1n_2}\Big)^{s_1} \Big(\frac{G}{4\pi d_2n_2}\Big)^{s_2} \Big(\frac{2n_2^2}{Gm_2^2}\Big)^{s_3} \cdot\nonumber \\
 & \times \int_0^\infty \frac{g(\sqrt{u})}{\sqrt{2\pi u}}u^{(1+s_1+s_2-s_3)/2}\cos\Big(u \frac{cG^2m_2}{8\pi m_1n_2^2}\Big) du ds_1 ds_2 ds_3\nonumber
\end{align}
Finally, we also perform an inverse Mellin transform on
$$\bar{g}_{1+s_1+s_2-s_3}^{\Re}(v):=\int_0^\infty \frac{g(\sqrt{u})}{\sqrt{2\pi u}} u^{(1+s_1+s_2-s_3)/2} \cos(uv) du$$
as indicated in equation \eqref{inverseMellin}. We arrive at
\begin{align*}
\mathcal{OD}=-\frac{\sqrt{2\pi} G^2}{16} &\frac{1}{(2\pi i)^4} \int_{(1+\varepsilon)}\int_{(1/2+3\varepsilon)}\int_{(2)}\int_{(1+3\varepsilon)} \tilde{\psi_1}(s_1)\tilde{\psi_2}(s_2)\Gamma(s_3)\tilde{\bar{g}}_{1+s_1+s_2-s_3}^{\Re}(s_4)\cdot \\
 & \times \sum_{d_1,d_2} \sum_{n_2} \sum_{m_1, m_2} \frac{1}{d_1d_2m_1n_2^2}\cdot \Big(\frac{\sqrt{u}Gm_2}{4\pi m_1d_1n_2}\Big)^{s_1} \Big(\frac{\sqrt{u}G}{4\pi d_2n_2}\Big)^{s_2} \Big(\frac{2n_2^2}{\sqrt{u}Gm_2^2}\Big)^{s_3}\nonumber\\
& \times{\sum_{c}} \frac{\phi(c)}{c}\Big(\frac{8\pi m_1n_2^2}{cG^2m_2}\Big)^{s_4}du ds_1ds_2ds_3ds_4. \nonumber
\end{align*}

Rewrite the various summations in terms of zeta functions to get
\begin{align}\label{offdiagonalcomplete3}
\mathcal{OD}=-\frac{\sqrt{2\pi} G^2}{16}&\frac{1}{(2\pi i)^4} \int_{(1+\varepsilon)}\int_{(1/2+3\varepsilon)}\int_{(2)}\int_{(1+3\varepsilon)} \tilde{\psi_1}(s_1)\tilde{\psi_2}(s_2)\Gamma(s_3)\tilde{\bar{g}}^{\Re}_{1+s_1+s_2-s_3}(s_4)\zeta(1+s_1)\\
&\times \zeta(1+s_2)\zeta(2+s_1+s_2-2s_3-2s_4)\zeta(1+s_1-s_4)\zeta(-s_1+2s_3+s_4)\nonumber \\
&\times \frac{\zeta(s_4)}{\zeta(1+s_4)}\big(1+(-1)^{-1-s_1+s_4}\big)\big(1+(-1)^{s_1-2s_3-s_4}\big)(4\pi)^{-s_1-s_2}2^{s_3}(8\pi)^{s_4} \nonumber\\
&\times G^{s_1+s_2-s_3-2s_4} ds_1d_2ds_3ds_4.\nonumber
\end{align}
The zeta functions $\zeta(1+s_1)$ and $\zeta(1+s_2)$ arise from summing over $d_1, d_2$
respectively. The summation over $n_2$ yields $\zeta(2+s_1+s_2-2s_3-2s_4)$, while
summing over $m_1$ gives rise to $\zeta(1+s_1-s_4)$. The $m_2$-variable leads
to the factor $\zeta(-s_1+2s_2+s_4)$ and finally the summation over $c$
gives $\zeta(s_4)/\zeta(1+s_4)$. Here we have used
the fact that $\sum_{c}\frac{\phi{c}}{c^s} =\frac{\zeta(s-1)}{\zeta(s)}$ for $Re(s)>2$.
To evaluate expression $\mathcal{OD}$ asymptotically we will
iteratively shift the contours and pick up poles.

We start to compute the contour integral \eqref{offdiagonalcomplete3}
by shifting the line from $\Re(s_2)=2$ to $\Re(s_2)=-100$.
We pick up a simple pole at $s_2=0$ and $s_2=-1-s_1+2s_3+2s_4$.
The integral over the new line  is negligible by the rapid decay
of $\tilde{\psi_1}, \tilde{\psi_2}, \tilde{\bar{g}}^{\Re}_{w}$
and the Gamma function. The residue at $s_2=0$ is given by
\begin{align}\label{offdiagonalcomplete5}
-\frac{\sqrt{2\pi} G^2}{16}&\frac{1}{(2\pi i)^3} \int_{(1+\varepsilon)}\int_{(1/2+3\varepsilon)}\int_{(1+3\varepsilon)} \tilde{\psi_1}(s_1)\tilde{\psi_2}(0)\Gamma(s_3)\tilde{\bar{g}}^{\Re}_{1+s_1-s_3}(s_4) \\
&\times \zeta(1+s_1)\zeta(2+s_1-2s_3-2s_4)\zeta(1+s_1-s_4)\zeta(-s_1+2s_3+s_4)\frac{\zeta(s_4)}{\zeta(1+s_4) }  \nonumber \\
&\times \big(1+(-1)^{-1-s_1+s_4}\big)\big(1+(-1)^{s_1-2s_3-s_4}\big)(4\pi)^{-s_1}2^{s_3}(8\pi)^{s_4} G^{s_1-s_3-2s_4} ds_1ds_3ds_4.\nonumber
\end{align}

Moving the line $\Re(s_1)=1+2\varepsilon$ to $\Re(s_1)=-100$ yields poles at $s_1=0$ and $s_1=s_4$.
Firstly, we consider the residue at $s_1=0$, which is given by
\begin{align}\label{offdiagonalcomplete6}
-\frac{\sqrt{2\pi} G^2}{16}&\frac{1}{(2\pi i)^2} \int_{(1+\varepsilon)}\int_{(1/2+2\varepsilon)} \tilde{\psi_1}(0)\tilde{\psi_2}(0)\Gamma(s_3)\tilde{\bar{g}}^{\Re}_{1-s_3}(s_4)\big(1+(-1)^{-1+s_4}\big)\big(1+(-1)^{-2s_3-s_4}\big) \\
&\times \zeta(2-2s_3-2s_4)\zeta(1-s_4)\zeta(2s_3+s_4)\frac{\zeta(s_4)}{\zeta(1+s_4) }
 2^{s_3}(8\pi)^{s_4} G^{-s_3-2s_4} ds_3ds_4\nonumber
\end{align}
which is negligible upon shifting $\Re(s_3)=1/2+3\varepsilon$ to $\Re(s_3)=100$
and by the rapid decay of $\tilde{\bar{g}}$ and the Gamma function.
On the other hand the residue of the pole at $s_1=s_4$ of \eqref{offdiagonalcomplete5} leads to
\begin{align}\label{offdiagonalcomplete7}
-\frac{\sqrt{2\pi} G^2}{16}&\frac{1}{(2\pi i)^2} \int_{(1+\varepsilon)}\int_{(1/2+2\varepsilon)}\tilde{\psi_1}(s_4)\tilde{\psi_2}(0)\Gamma(s_3)\tilde{\bar{g}}^{\Re}_{1+s_4-s_3}(s_4)\big(1+(-1)^{-1}\big)\big(1+(-1)^{-2s_3}\big)\\
&\times \zeta(2-2s_3-s_4)\zeta(2s_3)\zeta(s_4)  (4\pi)^{-s_4}2^{s_3}(8\pi)^{s_4} G^{-s_3-s_4} ds_3ds_4,\nonumber
\end{align}
which equals to $0$.
In total we find that the contribution from poles
that arise from $s_2=0$ is negligible.

We now evaluate the residue of  \eqref{offdiagonalcomplete5} at the pole $s_2=-1-s_1+2s_3+2s_4$. We get
\begin{align}\label{offdiagonalcomplete8}
-\frac{\sqrt{2\pi} G^2}{16}&\frac{1}{(2\pi i)^3} \int_{(1+\varepsilon)}\int_{(1/2+3\varepsilon)}\int_{(1+3\varepsilon)} \tilde{\psi_1}(s_1)\tilde{\psi_2}(-1-s_1+2s_3+2s_4)\Gamma(s_3)\tilde{\bar{g}}^{\Re}_{s_3+2s_4}(s_4) \\
&\times \zeta(1+s_1)\zeta(-s_1+2s_3+2s_4)\zeta(1+s_1-s_4)\zeta(-s_1+2s_3+s_4)\frac{\zeta(s_4)}{\zeta(1+s_4)}  \nonumber \\
&\times \big(1+(-1)^{-1-s_1+s_4}\big)\big(1+(-1)^{s_1-2s_3-s_4}\big)(4\pi)^{1-2s_3-2s_4}2^{s_3}(8\pi)^{s_4} G^{-1+s_3} ds_1ds_3ds_4.\nonumber
\end{align}
Next we move the line $\Re(s_3)=1/2+3\varepsilon$ to $\Re(s_3)=\varepsilon$
(stopping before the pole of the Gamma function) and capture poles
at  $s_3=1/2+s_1/2-s_4/2$. The integrals over the new lines
contribute at most $O(K^{2+\varepsilon}G^{-1})$. The residue at $s_3=1/2+s_1/2-s_4/2$ is given by
\begin{align}\label{offdiagonalcomplete12}
-\frac{\sqrt{2\pi} G^2}{16}&\frac{1}{(2\pi i)^2} \int_{(1+\varepsilon)}\int_{(1+3\varepsilon)} \tilde{\psi_1}(s_1)\tilde{\psi_2}(s_4)\Gamma(1/2+s_1/2-s_4/2)\tilde{\bar{g}}^{\Re}_{1/2+s_1/2+3/2s_4}(s_4) \\
&\times \zeta(1+s_1)\zeta(1+s_1-s_4)\frac{1}{2}\zeta(s_4)  (4\pi)^{-s_1-s_4}2^{1/2 +s_1/2-s_4/2}(8\pi)^{s_4} G^{-1/2+s_1/2-s_4/2} \nonumber\\
&\times \big(1+(-1)^{-1-s_1+s_4}\big)\big(1+(-1)^{-1}\big)ds_1ds_4,\nonumber
\end{align}
which equals $0$,
since $\psi_1(y)$ is even, i.e., $\psi_1(y)=\psi_1(1/y)$
implies $\tilde{\psi_1}(s)=\tilde{\psi_1}(-s)$.

Now we finish the proof by setting $\theta=\frac{9}{10},\delta=\frac{9}{160}, \eta=\frac{13}{80}$,
which satisfy all restrictions of $\theta,\delta,\eta$.
\end{proof}

\subsection{Proof of the main theorem}
\begin{proof}[Proof of Theorem \ref{mainthm}]
Recall the definition of $\mathcal{E}_{\psi}$ (see \eqref{errorterm})
and $\mathcal{S}_\psi$ (see \eqref{shiftedconvolutiondefi}).
By Luo and Sarnak (\cite[Section 5]{luoandsarnakMassEquidistributionHecke}) ,
\begin{equation}\label{LuoSarnak}\sum_{k\equiv 0\mod{2}} h\Big(\frac{k-1-K}{G}\Big) \sum_{f\in H_k} L(1, \sym^2 f) |E_\psi|^2 \ll G^{1+\varepsilon}.
\end{equation}
Moreover, we have
\begin{align*}
V(\psi_1, \psi_2) &=\sum_{k\equiv 0\mod{2}} h\Big(\frac{k-1-K}{G}\Big) \sum_{f\in H_k} L(1, \sym^2 f) (S_{\psi_1} + E_{\psi_1})(S_{\psi_2}+E_{\psi_2})\\
&=\sum_{k\equiv 0\mod{2}} h\Big(\frac{k-1-K}{G}\Big) \sum_{f\in H_k} L(1, \sym^2 f) (S_{\psi_1}S_{\psi_2}+S_{\psi_1}E_{\psi_2}+S_{\psi_2}E_{\psi_1}+E_{\psi_1}E_{\psi_2}).
\end{align*}
In Section \ref{VarianceComputation} we have proved that
$$\mathcal{M}(\psi_1, \psi_2)=\sum_{k\equiv 0\mod{2}}
h\Big(\frac{k-1-K}{G}\Big) \sum_{f\in H_k} L(1, \sym^2 f) S_{\psi_1}S_{\psi_2}=
\mathcal{D}+\mathcal{OD}.$$
Then Theorem \ref{mainthm} follows from the Cauchy-Schwarz inequality and the estimate in \eqref{LuoSarnak}.
\end{proof}

\subsection*{Acknowledgments.}
The authors would like to thank Yongxiao Lin for many valuable suggestions.

\bibliographystyle{alpha}
\bibliography{ref_1}

\end{document}